\documentclass[12pt]{article}

\usepackage[a4paper, margin=1in, footskip=0.3in]{geometry}

\usepackage{iftex}
\ifPDFTeX
  \usepackage[T1]{fontenc}
\fi
\usepackage{lmodern}
\usepackage{microtype} 

\usepackage{amsmath,amssymb,amsthm,mathtools}

\usepackage{enumitem}
\usepackage{graphicx}
\usepackage{tikz}
\usetikzlibrary{arrows.meta,calc}

\allowdisplaybreaks
\newcommand{\E}{\mathbb E}
\newcommand{\Pbb}{\mathbb P}
\newcommand{\eps}{\varepsilon}
\newcommand{\Core}{\mathbf C}
\newcommand{\Kernel}{\mathcal K}
\newcommand{\Kemeny}{t_{\odot}}
\newcommand{\meet}{\tau_{\rm meet}}
\newcommand{\cons}{T_{\rm cons}}
\newcommand{\coal}{T_{\rm coal}}
\newcommand{\tmix}{t_{\rm mix}}
\newcommand{\Egr}{\mathbb E_{\mathrm{gr}}}
\newcommand{\Vargr}{\operatorname{Var}_{\mathrm{gr}}}
\newcommand{\Reff}{R_{\rm eff}}
\DeclareMathOperator{\dist}{dist}
\DeclareMathOperator{\Var}{Var}

\newcommand{\cE}{\mathcal E}
\newcommand{\Diag}{\mathrm{Diag}}
\newcommand{\one}{\mathbf{1}}
\newcommand{\PGW}{\mathrm{PGW}}
\newcommand{\Poi}{\operatorname{Poisson}}
\newcommand{\TV}{\mathrm{TV}}
\newcommand{\medLarge}{\fontsize{16pt}{16pt}\selectfont}
\newtheorem{theorem}{Theorem}
\newtheorem{corollary}[theorem]{Corollary}
\newtheorem{lemma}[theorem]{Lemma}
\newtheorem{claim}{Claim}[theorem]

\usepackage[hidelinks]{hyperref}
\hypersetup{
  pdftitle={Meeting and coalescence times for random walks in the largest component of the Erd\H{o}s--R\'enyi random graph},
  pdfauthor={Vyacheslav Koval, Yuval Peres, Pieter Trapman},
  pdfsubject={Meeting times, coalescing random walks, and voter consensus on Erd\H{o}s--R\'enyi random graphs},
  pdfkeywords={Erd\H{o}s-R\'enyi random graph, meeting time, coalescing random walks, voter model, random walks on graphs, effective resistance}
}

\title{\medLarge\textbf{Meeting and coalescence times for random walks in the largest component of the Erd\H{o}s--R\'enyi random graph}}

\author{
  \textrm{Vyacheslav Koval}\thanks{Bernoulli Institute, University of Groningen, The Netherlands. \texttt{v.v.koval@rug.nl}}
  \and
  \textrm{Yuval Peres}\thanks{Beijing Institute of Mathematical Sciences and Applications (BIMSA), China. \texttt{yperes@gmail.com}}
  \and
  \textrm{Pieter Trapman}\thanks{Bernoulli Institute, University of Groningen, The Netherlands. \texttt{j.p.trapman@rug.nl}}
}
\date{\small July 11, 2026}

\begin{document}

\maketitle

\begin{abstract}
We prove that the stationary and worst-case expected meeting times of two
independent continuous-time random walks on the largest component of the Erd\H{o}s--R\'enyi random graph $G(n,p)$ have order $n$
throughout the strictly supercritical, the slightly supercritical and the critical regimes. Using these bounds along with a fine-tuned combination of comparison inequalities due to Oliveira (2012) and Kanade--Mallmann-Trenn--Sauerwald (KMS, 2023), we deduce that
expected coalescence time and full voter-model consensus also have order $n$ throughout these three regimes.
\vspace{0.5em}

\bigskip


\medskip

\end{abstract}

\section{Main results}\label{sec:main-results}

Let $G(n,p)$ be the Erd\H{o}s--R\'enyi graph on $[n]$, and write
$\mathcal C_1$ for its largest connected component. On any finite connected
component $G$, let $(X_t)_{t\ge0}$ and $(Y_t)_{t\ge0}$ be independent
continuous-time simple random walks with total jump rate $1$ and stationary law
$\pi_G$; when the component is clear, write $\pi=\pi_G$. We write
\[
        \meet=\inf\{t\ge0:X_t=Y_t\},
        \qquad
        t_{\rm meet}(G)=\max_{x,y\in V(G)}\E^G_{x,y}\meet .
\]
The expectation $\E^G_{\pi\otimes\pi}$ means that the two walks are started
independently from $\pi_G$. All random-walk expectations below are quenched,
that is, conditional on the sampled component. All limit statements in the
sequel are as $n\to\infty$; in particular, an event holds \emph{with high
probability} if its probability tends to one in this limit. A constant is called
\emph{absolute} if it is independent of $n$ and of the supercritical parameter,
including $\eps$ and the fixed value of $\lambda$.

\begin{theorem}[Main meeting-time bounds]\label{thm:main-results}
There are absolute constants $0<c<C<\infty$ such that the first two assertions
below hold; these constants do not depend on $\eps$ or on the fixed value of
$\lambda$. In all three parts, $n$ is the number of vertices of the original
Erd\H{o}s--R\'enyi graph.
\begin{enumerate}[label=\textup{(\roman*)}]
\item Let $\eps=\eps(n)\to0$ with $\eps^3n\to\infty$, and let $\mathcal C_1$ denote the
largest component of $G(n,(1+\eps)/n)$. Then, with high probability,
\[
        c n
        \le \E^{\mathcal C_1}_{\pi\otimes\pi}\meet
        \le t_{\rm meet}(\mathcal C_1)
        \le C n .
\]

\item Fix $\lambda>1$, and let $\mathcal C_1$ denote the largest component of
$G(n,\lambda/n)$. Then, with high probability,
\[
        c n
        \le \E^{\mathcal C_1}_{\pi\otimes\pi}\meet
        \le t_{\rm meet}(\mathcal C_1)
        \le C n .
\]

\item Fix $A_0\ge0$ and suppose that
$|n^{1/3}(np_n-1)|\le A_0$. Let $\mathcal C_1$ denote the largest component of
$G(n,p_n)$. Then, for every $\eta>0$, there are constants
$0<c_{A_0,\eta}<C_{A_0,\eta}<\infty$, depending only on $A_0$ and $\eta$, such
that, for all sufficiently large $n$,
\[
        \Pbb\left(
        c_{A_0,\eta} n
        \le \E^{\mathcal C_1}_{\pi\otimes\pi}\meet
        \le t_{\rm meet}(\mathcal C_1)
        \le C_{A_0,\eta} n
        \right)\ge1-\eta .
\]
\end{enumerate}
\end{theorem}

The theorem leads, after applying and extending comparison results of Oliveira
\cite{oliveira-coalescence} and Kanade--Mallmann-Trenn--Sauerwald (KMS)
\cite{kms}, to the following corollary.

\begin{corollary}[Voter consensus]\label{cor:voter}
Consider the continuous-time voter model on the component specified below,
where each vertex updates at rate $1$ by copying a uniformly chosen neighbor,
and suppose that the initial opinions are all distinct. Let $\cons$ be the
consensus time. There are absolute constants $0<c<C_{\rm sup}<\infty$ such that
the first two assertions below hold.
\begin{enumerate}[label=\textup{(\roman*)}]
\item In the fixed-supercritical case, let $H_\lambda$ be the giant component of
$G(n,\lambda/n)$ for a fixed $\lambda>1$. Then, with high probability,
\[
        c n \le \E^{H_\lambda}\cons \le C_{\rm sup} n.
\]

\item In the slightly supercritical case, let $H_{\rm ER}$ be the giant component
of $G(n,(1+\eps)/n)$, where $\eps\to0$ and $\eps^3n\to\infty$. Then, with high
probability,
\[
        c n \le \E^{H_{\rm ER}}\cons \le C_{\rm sup} n.
\]

\item In the critical window, fix $A_0\ge0$ and suppose that
$|n^{1/3}(np_n-1)|\le A_0$. Let $\mathcal C_1$ be the largest component of $G(n,p_n)$.
For every $\eta>0$ there are constants
$0<c_{A_0,\eta}<C_{A_0,\eta}<\infty$, depending only on $A_0$ and $\eta$, such
that, for all sufficiently large $n$,
\[
        \Pbb\left(
        c_{A_0,\eta} n
        \le \E^{\mathcal C_1}\cons
        \le C_{A_0,\eta} n
        \right)\ge1-\eta .
\]
\end{enumerate}
In all three cases, the same displayed bounds hold for the full coalescence
time of coalescing random walks started with one particle at every vertex.
\end{corollary}

The theorem concerns two walks; the corollary passes from meeting estimates to
coalescence and voter consensus. In the slightly supercritical regime, applying
Lemma~\ref{lem:ct-kms} directly from time zero would introduce an unnecessary
logarithmic loss. We avoid it by first using Oliveira's bound for the time to
reach a prescribed number of clusters \cite[Theorem~1.2]{oliveira-coalescence},
together with the diameter estimate for the slightly supercritical giant and
commute times, and only then using a small-particle form of the KMS comparison
\cite{kms}. In the critical window we use Oliveira's universal comparison between
coalescence and maximal hitting time \cite[Theorem~1.1]{oliveira-coalescence},
combined with the critical-window volume and resistance bounds.

\section{Background}

Let two independent continuous-time random walks run on the same finite
connected graph. Their \emph{meeting time} is the first time at which the two
walks occupy one vertex. Start instead with one walk at every vertex, and merge
particles when they meet. The time until only one particle remains is the
\emph{coalescence time}. The classical duality between the voter model and
coalescing random walks identifies this time with voter-model consensus when all
initial opinions are distinct \cite{liggett}.

We determine the scale of these times on the large components of the
Erd\H{o}s--R\'enyi graph $G(n,p)$ as $p$ crosses the transition point $p=1/n$.
There are three regimes. The \emph{critical window} is
$p=(1+O(n^{-1/3}))/n$. The \emph{slightly supercritical regime} is
$p=(1+\eps)/n$, where $\eps\to0$ but $\eps^3n\to\infty$. The \emph{fixed
supercritical regime} is $p=\lambda/n$ for fixed $\lambda>1$. In each case the
meeting-time scale is the ambient scale $n$: the vertex count of the original
random graph, not the typically smaller size of the component under study. In
the two supercritical meeting-time bounds the comparison constants are absolute;
they do not depend on $n$, $\eps$, or the fixed value of $\lambda$.

Meeting, coalescence, and consensus have a substantial history. Cox
\cite{cox} studied coalescing random walks and voter consensus on tori.
Oliveira \cite{oliveira-coalescence} proved a universal comparison between
coalescence and maximal hitting time, and Oliveira \cite{oliveira} gave
mean-field conditions under which coalescence is controlled by a stationary
two-walk meeting time. Kanade, Mallmann-Trenn and Sauerwald \cite{kms} related
meeting, mixing, and coalescence on general finite graphs. Cooper, Frieze and
Radzik \cite{cooper-frieze-radzik} studied multiple random walks on random
regular graphs; Avena et al. \cite{avena-capannoli-hazra-quattropani} treated
analogous questions on random directed graphs. There is also a large one-walk
literature on Erd\H{o}s--R\'enyi graphs: see, for example, Cooper and Frieze
\cite{cooper-frieze-cover} for the cover time of the fixed supercritical giant,
Barlow, Ding, Nachmias and Peres \cite{bdnp-cover} for the evolution of cover
time through criticality, and Ding, Lee and Peres
\cite{ding-lee-peres-cover} for the resistance-Gaussian description of cover
times. Meeting time is a genuinely two-walk quantity. In the supercritical
regimes considered here it is far smaller than cover time.

The upper bound starts from an electrical identity. Give every edge unit
conductance and let
\[
        \pi(x)=\frac{d(x)}{2|E(G)|}
\]
be the stationary distribution of the continuous-time walk with total jump rate
$1$. If $\tau_v$ is the hitting time of $v$ and $\E_\pi$ denotes a walk started
from $\pi$, the \emph{random target time}, or Kemeny's constant, is
\[
        \Kemeny(G):=\sum_v\pi(v)\E_\pi\tau_v.
\]
The commute identity gives the one-walk quantity the resistance form
\begin{equation}\label{eq:Kemenyfirst}
        \Kemeny(G)
        =|E(G)|\sum_{x,y}\pi(x)\pi(y)\Reff(x,y).
\end{equation}
Here $\Reff(x,y)$ is the effective electrical resistance between $x$ and $y$.
Equivalently, $\Reff(x,y)^{-1}$ is the minimum network energy of a function
that is $0$ at $x$ and $1$ at $y$. We use Thomson's principle and the
commute-time identity with this normalization; see Levin and Peres
\cite[Chapter~9]{lpw} and Lyons and Peres \cite[Chapters~2 and~9]{lyons-peres}.
Lemma~\ref{lem:ct-commute} checks the continuous-time constants.

For the giant of the slightly supercritical graph, the Ding--Kim--Lubetzky--Peres decomposition
\cite[Theorem~2]{dklp} gives three layers. The \emph{kernel} is a
configuration-model multigraph with minimum degree three. Replacing each kernel
edge by a geometrically distributed path produces the \emph{2-core}.
Independent subcritical Poisson Galton--Watson trees are then attached to the
core vertices. In this model the degree-weighted average resistance has order
$\eps^{-1}$, while the giant has $\Theta(\eps n)$ edges and vertices. These two
factors cancel in Kemeny's formula \eqref{eq:Kemenyfirst}, leaving an $O(n)$
random-target bound. Keeping the parameter dependence in the same calculation
gives the fixed-supercritical estimate. Aldous's comparison
\cite[Proposition~2]{aldous-meeting}, in the form stated in
Lemma~\ref{lem:aldous-meeting}, then turns the random-target and mixing bounds
into a worst-case meeting-time bound.

The lower bound is variational. We build a low-energy function on pairs of
vertices. It equals one once the two walks have met, but has small stationary
mean. The inequality in Lemma~\ref{lem:stationary-variational} then forces a
large expected meeting time of the walks. Near criticality, the function records
whether the two vertices project to nearby roots in the core $\Core$. Away
from criticality, the diagonal indicator already suffices. Section~\ref{sec:roadmap}
spells out this proof strategy.

In the critical window there is no linear-size giant, so component-volume
estimates are most naturally stated at fixed probability level. We use the
critical-window volume, diameter, mixing, and Nash--Williams resistance estimates
of Nachmias and Peres
\cite[Theorems~1.1 and~1.2, Lemma~5.4, Propositions~5.5--5.7]{nachmias-peres-critical}.
Those estimates give the $O(n)$ upper bound. They also give two sets of positive
stationary mass separated by resistance of order $n^{1/3}$, which yields the
matching lower bound.

We write DKLP for Ding, Kim, Lubetzky, and Peres \cite{dklp}. The other main
inputs are the strictly supercritical contiguous decomposition
\cite[Theorem~1]{dlp-strict}, the configuration-model expansion estimate
\cite[Lemma~5.3]{bkw}, the supercritical mixing estimates
\cite[Theorem~1 and Section~4.2]{dlp} and \cite[Theorem~1.1]{bkw}, the
slightly supercritical giant diameter estimate of Ding--Kim--Lubetzky--Peres
\cite{dklp-diameter}, and the critical-window estimates cited above. Here
\emph{contiguous} is used in the direction needed for the proof: any graph
property that holds with high probability in the auxiliary model holds with high
probability for the corresponding Erd\H{o}s--R\'enyi component. The weighted
resistance estimates, core-ball estimates, random-target bounds, and
variational lower bounds are proved below. The random-target lemma is stated
early because it drives the meeting-time upper bound; its proof is deferred
until Section~\ref{sec:remaining-proofs}, after the DKLP estimates have been
established.

\section{Notation, conventions, and auxiliary model}

We now fix the detailed conventions used in the proofs. The Erd\H{o}s--R\'enyi
graph $G(n,p)$ has vertex set $[n]$ and contains each unordered pair
independently with probability $p$. We write $\mathcal C_1$ for its largest
connected component; above the critical window this is the giant component. An
event holds \emph{with high probability} if its probability tends to one as
$n\to\infty$.

A constant is \emph{absolute} if it is independent of $n$ and of every
supercritical parameter, including $\eps$ and the fixed value of $\lambda$. Thus
the constants in the two supercritical parts of Theorem~\ref{thm:main-results}
are uniform over fixed $\lambda>1$. The failure probability in the phrase
``with high probability'' may still depend on that fixed value of $\lambda$.
Likewise, in fixed-supercritical statements, the threshold from which $n$ is
large enough may depend on the fixed value of $\lambda$, but constants denoted
$c,C,C_{\rm K},C_*,C_{\rm sup}$ and similar absolute constants do not.
Dependence on other displayed parameters is written explicitly, as in
$C_{A_0,\eta}$. The symbol $\one_A$ denotes the indicator of an event or set
$A$.

For a finite connected multigraph $G$, write $|E(G)|$ for its number of
edges. We view $G$ as a unit-conductance network. Degrees count incident half-edges; a self-loop
contributes two. At an update time the walk chooses one incident half-edge
uniformly. If the chosen half-edge is a loop, the walk remains at its current
vertex. In the generator this contributes no difference $f(y)-f(x)$, and in the
network it has no effect on effective resistance. Multiple non-loop edges are
parallel unit resistors, and edge sums count multiplicity.

Let $(X_t)_{t\ge0}$ be the continuous-time simple random walk on $G$ with total jump rate
$1$ at every vertex. Its generator is
\begin{equation}\label{eq:generator}
        \mathcal L_G f(x)=\frac1{d(x)}\sum_{y\sim x}\bigl(f(y)-f(x)\bigr),
\end{equation}
where $d(x)$ denotes degree in the multigraph convention above. Its stationary
measure is
\begin{equation}\label{eq:pi}
        \pi_G(x)=\frac{d(x)}{2|E(G)|}.
\end{equation}
When the graph is clear from context we write $\mathcal L$ and $\pi$. Run two
independent copies $(X_t)_{t\ge0}$ and $(Y_t)_{t\ge0}$, and define
\begin{equation}\label{eq:meet}
        \meet:=\inf\{t\ge0:X_t=Y_t\}.
\end{equation}
For a vertex $v$, write $\tau_v:=\inf\{t\ge0:X_t=v\}$ for the hitting time of
$v$ by one walk. The worst-case expected meeting time is
\begin{equation}\label{eq:graph-tmeet}
        t_{\rm meet}(G):=\max_{x,y\in V(G)}\E^G_{x,y}\meet .
\end{equation}

All random-walk probabilities and expectations are quenched. Once a graph $G$
has been sampled, $\Pbb^G_{x,y}$ and $\E^G_{x,y}$ denote probability and
expectation for two walks on this fixed graph. Similarly, $\E^G_\pi$ denotes
expectation for one walk started from stationarity, and
$\E^G_{\pi\otimes\pi}$ denotes expectation for two independent stationary walks.
When the graph is clear from context, $\Pbb_x$, $\E_x$, $\Pbb_\pi$ and $\E_\pi$
have the corresponding one-walk meanings. The same quenched notation $\E^G$ is
used below for the voter model and for coalescing random walks on the fixed
graph $G$. The unadorned symbols $\Pbb$ and $\Egr$ refer to graph-construction
probability and expectation.

For vertices $x,y$ in a connected multigraph $G$, let $\Reff^G(x,y)$ be the
effective resistance; we omit the superscript when the graph is clear. For
nonempty disjoint sets $A,B\subseteq V(G)$, $\Reff^G(A,B)$ denotes the effective
resistance between the two terminals obtained by identifying all vertices in
$A$ and all vertices in $B$. Equivalently, $\Reff^G(A,B)^{-1}$ is the minimum of
the unnormalized network energy
\[
        \mathcal D_G(f):=\sum_{xy\in E(G)}(f(x)-f(y))^2
\]
over functions that are $0$ on $A$ and $1$ on $B$, with non-loop edges counted
with multiplicity and loops contributing zero. The one-walk Dirichlet form for
the jump-rate-$1$ walk is the normalized form
$(2|E(G)|)^{-1}\mathcal D_G(f)$.

For a finite connected multigraph $G$, let $\Delta_G$ denote the symmetric
unnormalized graph Laplacian. We write $\Delta_G^{\dagger}$ for its
Moore--Penrose inverse, equivalently the inverse on the subspace orthogonal to
constants and zero on the constants \cite{penrose}.

Let $P^t=e^{t\mathcal L_G}$ be the transition semigroup of the continuous-time
walk. Our total-variation mixing time convention is
\begin{equation}\label{eq:tmix-convention}
        \tmix(G):=\inf\left\{t\ge0:\max_x\|P^t(x,\cdot)-\pi_G\|_{\TV}\le\frac14\right\},
\end{equation}
for the jump-rate-$1$ continuous-time walk unless stated otherwise. Changing the
threshold changes the estimates only by absolute multiplicative constants.

We next specify the slightly supercritical auxiliary model used in the proof.
Let $\eps=\eps(n)\to0$ and assume
\[
        a_n:=\eps^3 n\to\infty.
\]
This is the slightly supercritical regime, just outside the critical window. The
giant component has appeared, but the DKLP core-and-tree structure remains
visible at the scale used here. The condition $a_n\to\infty$ is also the range
in which the DKLP contiguous model \cite{dklp} and the Ding--Lubetzky--Peres
\cite{dlp} mixing estimate used below are available.

Let $H_{\rm ER}$ be the giant component of $G(n,(1+\eps)/n)$. We first prove the
main estimate for the DKLP contiguous model, then transfer it to $H_{\rm ER}$
by contiguity. Let $\mu\in(0,1)$ be the unique conjugate parameter associated
with $1+\eps$, namely the unique solution of
\begin{equation}\label{eq:mu-conjugate}
        \mu e^{-\mu}=(1+\eps)e^{-(1+\eps)}.
\end{equation}
Then $1-\mu\asymp\eps$. Here and below, $A_n\asymp B_n$ means that
$A_n/B_n$ is bounded above and below by positive constants in the regime under
consideration. We write $X_n=O_{\Pbb}(b_n)$ when $X_n/b_n$ is tight: for every
$\eta>0$ there are $K<\infty$ and $n_0$ such that
$\sup_{n\ge n_0}\Pbb(|X_n|>Kb_n)<\eta$.

By the contiguity theorem of Ding--Kim--Lubetzky--Peres
\cite[Theorem~2]{dklp}, the giant component of $G(n,(1+\eps)/n)$ is contiguous
to the following labelled model. This model is a concrete surrogate for the
giant: every isomorphism-invariant graph property that holds with high
probability in the model holds with high probability for the actual
Erd\H{o}s--R\'enyi giant.

Recall that the $2$-core of a graph is its maximal subgraph of minimum degree at
least two. Suppressing every maximal path whose internal vertices have degree
two produces the kernel. The DKLP construction reverses this procedure: sample
the kernel, restore the suppressed paths, and attach trees to the resulting
core.

First sample the DKLP parameter
\[
        \Lambda\sim\mathcal N\!\left(1+\eps-\mu,\frac1{\eps n}\right),
\]
as in \cite[Theorem~2]{dklp}. Since $1+\eps-\mu\asymp\eps$, a Gaussian tail
bound gives
\[
        \Pbb(\Lambda\le0)\le \exp(-c\eps^3n)=o(1).
\]
On this exceptional event, define the output of the auxiliary construction to
be a fixed graph consisting of two vertices joined by one edge and stop. This
completion changes the model only on an $o(1)$ event and therefore does not
affect any high-probability statement or the required direction of contiguity.
On $\{\Lambda>0\}$ continue with the construction below. We use only
\[
        \Lambda=(1+\eps-\mu)+O_{\Pbb}((\eps n)^{-1/2})
\]
and the degree moment bounds recorded in Lemma~\ref{lem:kernel-inputs}.
Conditional on $\Lambda$, take independent
$D_1,\ldots,D_n\sim\Poi(\Lambda)$, condition on
$\sum_i D_i\one_{\{D_i\ge3\}}$ being even, and put
$N_k=|\{i:D_i=k\}|$ for $k\ge3$. The kernel $\Kernel$ is the
configuration-model multigraph with $N_k$ vertices of degree $k$. Concretely,
one attaches $D_i$ half-edges to vertex $i$ and pairs all half-edges uniformly
at random; a pair at one vertex gives a self-loop, while repeated pairs give
parallel edges.

Second, replace each kernel edge by an independent path of length $L$, where
\[
        \Pbb(L=\ell)=(1-\mu)\mu^{\ell-1},\qquad \ell\ge1.
\]
This gives the core $\Core$, i.e. the $2$-core in the DKLP
construction. We use the term core $\Core$ throughout after this point. A kernel self-loop of stretch length $L\ge2$ is interpreted as a
cycle of length $L$ attached at its kernel vertex; a length-one self-loop
remains a self-loop under the multigraph conventions above.

Third, attach independent Poisson Galton--Watson trees with offspring mean
$\mu$, denoted $\PGW(\mu)$ trees, to the vertices of $\Core$. Each vertex of an
attached tree has an independent $\Poi(\mu)$ number of children.
Figure~\ref{fig:dklp} summarizes the construction and fixes the notation used
below. Sections~\ref{sec:input} and~\ref{sec:proof} work on a graph $H$ sampled
from this DKLP model; Section~\ref{sec:remaining-proofs} transfers the result
to the actual giant.

With high probability,
\begin{equation}\label{eq:sizes-intro}
        |V(H)|=\Theta(\eps n),\qquad
        |V(\Kernel)|=\Theta(\eps^3n),\qquad
        |V(\Core)|=\Theta(\eps^2n),\qquad
        |E(H)|=\Theta(\eps n).
\end{equation}

The contiguity arguments take place on the countable measurable space
$\mathfrak G$ of isomorphism classes of all finite multigraphs, including
empty and disconnected multigraphs, with no bound on their number of
vertices. For every nonempty connected element we use the walk convention
fixed above. This avoids discarding the negligible but possible events that
the auxiliary construction has more than $n$ vertices or is not connected.
A simple Erd\H{o}s--R\'enyi component is embedded in $\mathfrak G$ as a
connected multigraph without loops or multiple edges. The failure events used
for contiguity transfer are therefore genuine isomorphism-invariant graph
properties. Length-one self-loops
only create self-transitions and do not affect the Dirichlet or resistance
estimates. Stretched self-loop cycles in the fixed-supercritical model are
handled in the proof of Lemma~\ref{lem:kemeny}; their stationary contribution is
negligible.

Unless stated otherwise, constants in $O(\cdot)$ and $\Theta(\cdot)$ are
absolute in the slightly supercritical regime. The notation $O_r(\cdot)$ permits dependence
only on the displayed exponent $r$. We use $O_\lambda(\cdot)$,
$\Theta_\lambda(\cdot)$, and $o_\lambda(1)$ only for auxiliary estimates whose
constants may depend on the fixed parameter. This auxiliary notation does not
apply to the constants $c,C$ in Theorem~\ref{thm:main-results}. The notation
$O_{\Pbb,\lambda}(1)$ denotes a tight family with a bound allowed to depend on
that fixed $\lambda$.

\begin{figure}[!htbp]
\centering
\resizebox{0.68\textwidth}{!}{%
\begin{tikzpicture}[x=1cm,y=1cm,>=Latex,
    K/.style={circle,fill=black,inner sep=3.0pt},
    C/.style={circle,draw=black,fill=white,line width=0.45pt,inner sep=2.0pt},
    T/.style={circle,fill=black!38,inner sep=1.5pt},
    coreedge/.style={draw=black,line width=0.65pt},
    treeedge/.style={draw=black!42,line width=0.45pt},
    ann/.style={font=\small,fill=white,fill opacity=0.96,text opacity=1,inner sep=2pt,rounded corners=1pt}]

\node[K] (k1) at (0.0,0.0) {};
\node[K] (k2) at (1.8,1.25) {};
\node[K] (k3) at (4.1,1.55) {};
\node[K] (k4) at (7.0,0.95) {};
\node[K] (k5) at (7.9,-0.75) {};
\node[K] (k6) at (5.2,-1.10) {};
\node[K] (k7) at (2.55,-0.70) {};

\node[C] (c12a) at (0.55,0.35) {}; \node[C] (c12b) at (1.08,0.72) {}; \node[C] (c12c) at (1.45,1.02) {};
\draw[coreedge] (k1)--(c12a)--(c12b)--(c12c)--(k2);

\node[C] (c23a) at (2.35,1.35) {}; \node[C] (c23b) at (2.95,1.47) {}; \node[C] (c23c) at (3.55,1.55) {};
\draw[coreedge] (k2)--(c23a)--(c23b)--(c23c)--(k3);

\node[C] (c34a) at (4.75,1.47) {}; \node[C] (c34b) at (5.35,1.33) {}; \node[C] (c34c) at (5.95,1.18) {}; \node[C] (c34d) at (6.48,1.05) {};
\draw[coreedge] (k3)--(c34a)--(c34b)--(c34c)--(c34d)--(k4);

\node[C] (c45a) at (7.25,0.45) {}; \node[C] (c45b) at (7.55,-0.10) {}; \node[C] (c45c) at (7.78,-0.45) {};
\draw[coreedge] (k4)--(c45a)--(c45b)--(c45c)--(k5);

\node[C] (c56a) at (7.25,-0.85) {}; \node[C] (c56b) at (6.55,-0.98) {}; \node[C] (c56c) at (5.85,-1.06) {};
\draw[coreedge] (k5)--(c56a)--(c56b)--(c56c)--(k6);

\node[C] (c67a) at (4.55,-1.02) {}; \node[C] (c67b) at (3.85,-0.92) {}; \node[C] (c67c) at (3.20,-0.82) {};
\draw[coreedge] (k6)--(c67a)--(c67b)--(c67c)--(k7);

\node[C] (c71a) at (1.95,-0.55) {}; \node[C] (c71b) at (1.30,-0.38) {}; \node[C] (c71c) at (0.65,-0.18) {};
\draw[coreedge] (k7)--(c71a)--(c71b)--(c71c)--(k1);

\node[C] (c27a) at (1.95,0.55) {}; \node[C] (c27b) at (2.15,0.05) {}; \node[C] (c27c) at (2.35,-0.45) {};
\draw[coreedge] (k2)--(c27a)--(c27b)--(c27c)--(k7);

\node[C] (c36a) at (4.35,0.95) {}; \node[C] (c36b) at (4.65,0.35) {}; \node[C] (c36c) at (4.95,-0.35) {};
\draw[coreedge] (k3)--(c36a)--(c36b)--(c36c)--(k6);

\node[C] (c64a) at (5.85,-0.65) {}; \node[C] (c64b) at (6.35,-0.20) {}; \node[C] (c64c) at (6.75,0.42) {};
\draw[coreedge] (k6)--(c64a)--(c64b)--(c64c)--(k4);

\node[T] (l1) at (0.95,-1.00) {}; \node[T] (l2) at (0.58,-1.45) {}; \node[T] (l3) at (0.20,-1.80) {}; \node[T] (l4) at (0.80,-1.90) {}; \node[T] (l5) at (1.18,-1.42) {};
\draw[treeedge] (c71b)--(l1)--(l2)--(l3); \draw[treeedge] (l2)--(l4); \draw[treeedge] (l1)--(l5);

\node[T] (ul1) at (1.55,1.90) {}; \node[T] (ul2) at (1.35,2.45) {}; \node[T] (ul3) at (1.78,2.75) {}; \node[T] (ul4) at (1.15,2.78) {};
\draw[treeedge] (k2)--(ul1)--(ul2); \draw[treeedge] (ul2)--(ul3); \draw[treeedge] (ul2)--(ul4);

\node[T] (s1) at (2.95,2.10) {}; \node[T] (s2) at (2.65,2.45) {}; \node[T] (s3) at (3.20,2.48) {};
\draw[treeedge] (c23b)--(s1)--(s2); \draw[treeedge] (s1)--(s3);

\node[T] (i1) at (4.20,0.15) {}; \node[T] (i2) at (3.85,0.48) {}; \node[T] (i3) at (3.72,0.05) {}; \node[T] (i4) at (4.08,-0.30) {};
\draw[treeedge] (c36b)--(i1)--(i2); \draw[treeedge] (i1)--(i3); \draw[treeedge] (i1)--(i4);

\node[T] (um1) at (5.30,1.95) {}; \node[T] (um2) at (4.95,2.35) {}; \node[T] (um3) at (5.62,2.35) {}; \node[T] (um4) at (5.80,2.85) {};
\draw[treeedge] (c34b)--(um1)--(um2); \draw[treeedge] (um1)--(um3)--(um4);

\node[T] (ru1) at (6.70,1.75) {}; \node[T] (ru2) at (6.45,2.20) {}; \node[T] (ru3) at (6.85,2.55) {}; \node[T] (ru4) at (7.05,2.05) {};
\draw[treeedge] (c34d)--(ru1)--(ru2); \draw[treeedge] (ru2)--(ru3); \draw[treeedge] (ru1)--(ru4);

\node[T] (rr1) at (7.85,-1.30) {}; \node[T] (rr2) at (8.30,-1.55) {}; \node[T] (rr3) at (8.70,-1.25) {}; \node[T] (rr4) at (8.95,-1.65) {}; \node[T] (rr5) at (8.50,-2.02) {};
\draw[treeedge] (c56a)--(rr1)--(rr2)--(rr3); \draw[treeedge] (rr3)--(rr4); \draw[treeedge] (rr2)--(rr5);

\node[T] (bm1) at (5.05,-1.75) {}; \node[T] (bm2) at (4.75,-2.15) {}; \node[T] (bm3) at (5.30,-2.25) {}; \node[T] (bm4) at (4.90,-2.70) {}; \node[T] (bm5) at (5.65,-2.65) {};
\draw[treeedge] (k6)--(bm1)--(bm2)--(bm4); \draw[treeedge] (bm1)--(bm3)--(bm5);

\node[T] (bl1) at (3.80,-1.55) {}; \node[T] (bl2) at (3.45,-1.95) {}; \node[T] (bl3) at (4.15,-1.95) {};
\draw[treeedge] (c67b)--(bl1)--(bl2); \draw[treeedge] (bl1)--(bl3);

\node[ann] (labK) at (-0.65,1.55) {kernel $\Kernel$};
\draw[->,thin] (labK.east) .. controls (-0.05,1.25) and (0.05,0.45) .. (k1);

\node[ann] (labC) at (5.10,3.15) {core $\Core$ (2-core)};
\draw[->,thin] (labC.south) .. controls (5.05,2.55) and (5.25,2.05) .. (c34b);

\node[ann] (labT) at (4.55,-3.55) {independent attached $\mathrm{PGW}(\mu)$ trees};
\draw[->,thin] (labT.north) .. controls (4.70,-3.05) and (4.90,-2.55) .. (bm2);

\end{tikzpicture}}
\caption{Schematic of the Ding--Kim--Lubetzky--Peres configuration-model construction of the slightly supercritical giant. Black vertices form the kernel $\Kernel$. Replacing each kernel edge by an independent path of geometric length produces the core $\Core$, i.e. the $2$-core in the DKLP construction. The new degree-$2$ vertices created by stretching are shown in white. Independent attached $\mathrm{PGW}(\mu)$ trees are then rooted at vertices of $\Core$; their vertices are shown in gray.}
\label{fig:dklp}
\end{figure}
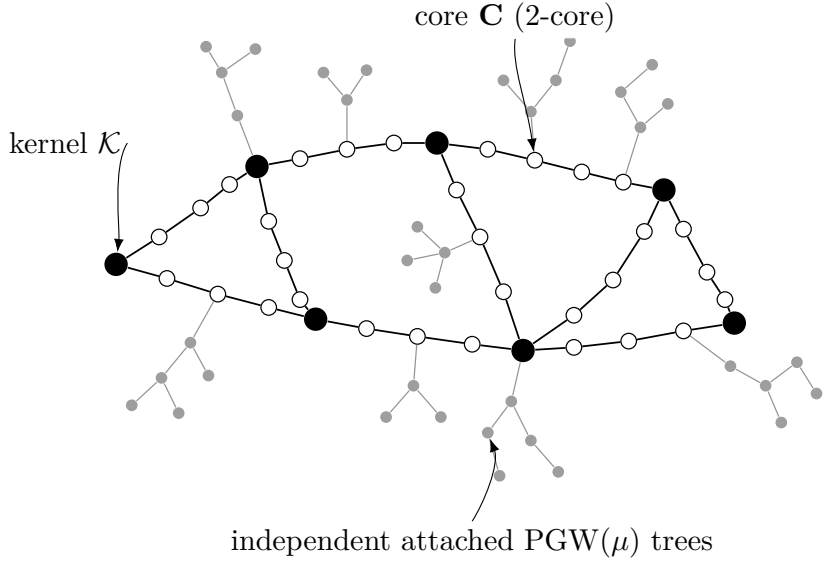

The following lemma is the technical estimate on the DKLP contiguous model. It is
where the geometry of the slightly supercritical giant enters the proof; the first part of
Theorem~\ref{thm:main-results} is obtained from it by contiguity.

\begin{lemma}[DKLP model meeting-time estimate]\label{lem:dklp-meeting}
There are constants $0<c<C<\infty$, independent of $n$ and $\eps$, such that,
with high probability, the DKLP model $H$ is connected and satisfies
\begin{equation}\label{eq:main-bounds}
        c n
        \le \E^H_{\pi\otimes\pi}\meet
        \le \max_{x,y\in V(H)}\E^H_{x,y}\meet
        \le Cn.
\end{equation}
\end{lemma}

The middle inequality in \eqref{eq:main-bounds} is immediate, since the
stationary expectation is an average of the quenched quantities
$\E^H_{x,y}\meet$ over starting pairs.

For a finite connected multigraph $G$, define the random target time by
\begin{equation}\label{eq:kemeny-target}
        \Kemeny(G):=\sum_{v\in V(G)}\pi(v)\E^G_\pi\tau_v .
\end{equation}
This is Kemeny's constant. The random target lemma \cite[Lemma~10.1]{lpw} says
that
\[
        \sum_{v\in V(G)}\pi(v)\E^G_x\tau_v=\Kemeny(G)
        \qquad\text{for every }x\in V(G).
\]
Lemma~\ref{lem:ct-commute} gives the equivalent resistance form
\begin{equation}\label{eq:kemeny-resistance}
        \Kemeny(G)=|E(G)|\sum_{u,v\in V(G)}\pi(u)\pi(v)\Reff^G(u,v).
\end{equation}
For the effective-resistance formulation of Kemeny's constant, see also
\cite{wang-dubbeldam-vanmieghem}.

\begin{lemma}[Uniform random-target bounds]\label{lem:kemeny}
There is an absolute constant $C_{\rm K}<\infty$ such that the following
estimates hold.
\begin{enumerate}
\item If $\eps=\eps(n)\to0$ and $\eps^3n\to\infty$, and $H_{\rm ER}$
denotes the giant component of $G(n,(1+\eps)/n)$, then, with high probability,
\[
        \Kemeny(H_{\rm ER})\le C_{\rm K} n .
\]
\item For every fixed $\lambda>1$, let $\mathcal C_1$ denote the giant component of
$G(n,\lambda/n)$. Then, with high probability,
\[
        \Kemeny(\mathcal C_1)\le C_{\rm K} n .
\]
\end{enumerate}
The same constant $C_{\rm K}$ works in both estimates; in particular it is
independent of the fixed value of $\lambda>1$.
\end{lemma}

\paragraph{Ambient normalization.}
The right side in Lemma~\ref{lem:kemeny} is measured in the ambient number of
vertices $n$, not in the component size $|\mathcal C_1|$. This distinction matters near
the transition. As $\lambda\downarrow1$, the giant has order
$(\lambda-1)n$ vertices and edges, while its degree-weighted average effective
resistance is larger by the compensating factor $(\lambda-1)^{-1}$.

The next three lemmas fix the normalizations for commute times, meeting times,
and variational hitting-time bounds.

\begin{lemma}[Continuous-time commute and target identities]\label{lem:ct-commute}
Let $G$ be a finite connected multigraph, viewed as a unit-conductance
network under the conventions above. Let $(X_t)_{t\ge0}$ be a continuous-time simple
random walk with total jump rate $1$. Then, for all vertices $u,v$,
\begin{equation}\label{eq:ct-commute-general}
        \E_u\tau_v+\E_v\tau_u=2|E(G)|\Reff(u,v).
\end{equation}
Consequently,
\begin{equation}\label{eq:ct-target-general}
        \Kemeny(G) = \sum_v\pi(v)\E_\pi\tau_v
        =|E(G)|\sum_{u,v}\pi(u)\pi(v)\Reff(u,v).
\end{equation}
\end{lemma}

\begin{proof}
The embedded jump chain is discrete-time simple random walk, and the holding
times between successive jumps are i.i.d.\ $\operatorname{Exp}(1)$. Conditional
on the number of jumps made before hitting $v$, the expected elapsed time equals
that number. Hence the continuous-time hitting expectation equals the expected
jump count of the embedded chain. The usual commute identity for the discrete
walk on an unweighted network \cite[Chapter~9]{lpw} therefore gives
\eqref{eq:ct-commute-general}.
\end{proof}

\begin{lemma}[Aldous meeting-time comparison]\label{lem:aldous-meeting}
Let $(X_t)_{t\ge0}$ be a finite irreducible reversible continuous-time chain
with stationary law $\pi$, and write $P^t$ for its transition semigroup. Let
$\tau_{\rm meet}$ be the first time two independent copies of the chain occupy
the same state. Put
\begin{equation}\label{eq:aldous-t1}
        t^*_{\rm mix}:=\inf\left\{t\ge0:
        \max_x\|P^t(x,\cdot)-\pi\|_{\TV}\le \frac1{2e}\right\},
\end{equation}
and let
\[
        T_M:=\max_{x,y}\E_{x,y}\tau_{\rm meet} .
\]
If $H_i$ denotes the hitting time of state $i$ by one copy of the chain, then
there is a universal constant $K_A<\infty$ such that
\begin{equation}\label{eq:aldous-prop2}
        T_M
        \le
        K_A\left(\sum_i
        \frac{\pi_i}{t^*_{\rm mix}\vee \E_\pi H_i}\right)^{-1} .
\end{equation}
It follows that, for continuous-time simple random walk with jump rate $1$ on a
finite connected multigraph $G$,
\begin{equation}\label{eq:aldous-kemeny-consequence}
        t_{\rm meet}(G)
        \le K'_A\bigl(\tmix(G)+\Kemeny(G)\bigr)
\end{equation}
for another universal constant $K'_A$.
\end{lemma}

\begin{proof}
Estimate \eqref{eq:aldous-prop2} is Proposition~2 of Aldous
\cite[Proposition~2]{aldous-meeting}. Aldous--Fill record the same comparison,
slightly rearranged, in \cite[Chapter~14, Corollary~14.7]{af}. Set
\[
        A_i:=t^*_{\rm mix}\vee \E_\pi H_i .
\]
Since $x\mapsto x^{-1}$ is convex,
\[
        \left(\sum_i\frac{\pi_i}{A_i}\right)^{-1}
        \le \sum_i\pi_iA_i
        \le t^*_{\rm mix}+\sum_i\pi_i\E_\pi H_i .
\]
For the jump-rate-$1$ simple random walk, the last sum is $\Kemeny(G)$ by
\eqref{eq:kemeny-target}. The threshold $1/(2e)$ differs from the threshold in
\eqref{eq:tmix-convention} only by an absolute factor, by submultiplicativity of total-variation distance
\cite[Chapter~4]{lpw}. Absorbing this factor into the constant gives
\eqref{eq:aldous-kemeny-consequence}.
\end{proof}

\begin{lemma}[Variational lower bound from stationarity]
\label{lem:stationary-variational}
Let $(Z_t)_{t\ge0}$ be a finite irreducible reversible continuous-time chain with
stationary law $\pi$ and generator $\mathcal L$. Write
$\cE(f,f):=-\langle f,\mathcal L f\rangle_\pi$ for its Dirichlet form. If
$B$ is a subset of the state space and $0\le F\le1$ satisfies $F=1$ on $B$, then
\begin{equation}\label{eq:stationary-variational}
        \E_\pi\tau_B\ge \frac{(1-\E_\pi F)^2}{\cE(F,F)} .
\end{equation}
\end{lemma}

\begin{proof}
This is the extremal characterization of hitting times from stationarity; see
Aldous--Fill \cite[Chapter~3, Proposition~3.41]{af}. With their normalization,
\[
        \E_\pi\tau_B
        =
        \sup\left\{
        \frac1{\cE(g,g)}:\ g=1 \text{ on } B,\ \E_\pi g=0
        \right\}.
\]
Taking
\[
        g=\frac{F-\E_\pi F}{1-\E_\pi F}
\]
gives \eqref{eq:stationary-variational}.
\end{proof}

\section{Proof roadmap}\label{sec:roadmap}

The upper bound has two deterministic ingredients. Lemma~\ref{lem:ct-commute}
identifies the random target time $\Kemeny(G)$ with a degree-weighted average
effective resistance. Lemma~\ref{lem:aldous-meeting} bounds
$t_{\rm meet}(G)$ by $\tmix(G)+\Kemeny(G)$. The random-graph work is therefore
focused: prove a uniform bound on $\Kemeny$.

For the giant of the slightly supercritical graph, attached trees contribute their average depth. The
core $\Core$ contributes average resistance of order $\eps^{-1}$. Since the
edge count is $\Theta(\eps n)$, the random target time is $O(n)$. The same
calculation, with the conjugate parameter kept throughout, gives the fixed
strictly supercritical bound with the same absolute comparison constant.

The lower bound uses the core projection. For $x\in V(H)$ let $r(x)$ be the
root of the attached tree containing $x$, viewed as a vertex of the stretched
core. The test function
\[
        F(x,y)=\max\left\{1-
        \frac{\dist_{\Core}(r(x),r(y))}{\lfloor\eps^{-1}\rfloor},0\right\}
\]
is one on the diagonal and constant inside attached trees. Its stationary mean
is $o(1)$: two independent stationary projected roots rarely lie within distance
$O(\eps^{-1})$. Its Dirichlet energy is $O(1/n)$, because only core moves can
change the function, and one core move changes it by $O(\eps)$. The variational
bound for hitting the diagonal gives $\E_{\pi\otimes\pi}\tau_{\rm meet}\ge cn$.

For fixed $\lambda>1$, Aldous's comparison (Lemma~\ref{lem:aldous-meeting}) and the uniform random-target
estimate give the upper bound. The lower constant must remain uniform in
$\lambda$, so the proof splits into two ranges. For $\lambda$ close to one, the same
core-projection test is uniform in the conjugate parameter. Once $\lambda$ is
bounded away from one, the diagonal test $F(x,y)=\one_{\{x=y\}}$ has stationary
mean and Dirichlet energy at most an absolute constant times $1/n$.

The critical-window statement uses the Nachmias--Peres estimates cited in
Lemma~\ref{lem:critical-geometry}. Their volume, diameter and mixing bounds give
the upper bound. Their Nash--Williams lower-bound argument gives two sets of
positive stationary mass separated by resistance of order $n^{1/3}$. Since the
component has $\Theta(n^{2/3})$ edges in a bounded critical window, edge volume
times resistance again gives the ambient scale $n$.

Section~\ref{sec:input} proves the slightly supercritical giant geometric inputs. Section~\ref{sec:proof}
combines those inputs with deterministic Markov-chain estimates to prove
Lemma~\ref{lem:dklp-meeting}. Section~\ref{sec:remaining-proofs} transfers the
result to Erd\H{o}s--R\'enyi graphs, treats the fixed-supercritical and
critical-window regimes, and proves the voter/coalescence corollary.

\section{Geometric input from the slightly supercritical giant}\label{sec:input}

We verify the random-graph inputs needed for the DKLP model of the slightly supercritical giant. Throughout
this section $H$ is the contiguous DKLP model, and all high-probability estimates
are uniform in the slightly supercritical regime.

For $u\in V(H)$ let $r(u)\in V(\Core)$ be the root of the attached tree
containing $u$. If $u\in V(\Core)$, set $r(u)=u$. Let
$h(u)=\dist(u,r(u))$. For $z\in V(\Core)$ put
\begin{equation}\label{eq:Wnu}
        W_z:=\sum_{u:r(u)=z}d(u),\qquad
        S:=\sum_{z\in V(\Core)}W_z=2|E(H)|,
        \qquad
        \nu(z):=\frac{W_z}{S}.
\end{equation}
If $U\sim\pi$, then $r(U)\sim\nu$. For $R\ge0$, define the core vertex ball,
its incident-edge count, and its degree volume by
\begin{equation}\label{eq:BE}
\begin{aligned}
        B_{\Core}(z,R)
        &:=\{w\in V(\Core):\dist_{\Core}(z,w)\le R\},\\
        B_E(z,R)
        &:=\bigl|\{e\in E(\Core):
        e\text{ has at least one endpoint in }B_{\Core}(z,R)\}\bigr|,\\
        B_D(z,R)
        &:=\sum_{w\in B_{\Core}(z,R)}d_{\Core}(w).
\end{aligned}
\end{equation}
Edges in $B_E$ are counted with multiplicity. Thus $B_E(z,R)$ includes edges
crossing from $B_{\Core}(z,R)$ to its complement; it is not merely the number of
edges induced by the vertex ball. In particular,
\[
        B_E(z,R)\le B_D(z,R)\le 2B_E(z,R).
\]

We collect the geometric facts about the slightly supercritical giant before proving the theorem.
The first estimates are elementary branching-process bounds for the attached
trees. They control two quantities used repeatedly: the degree mass attached to
a core vertex and the average depth seen from stationarity.

\begin{lemma}[Subcritical Galton--Watson moments]\label{lem:pgw-moments}
Let $0<\mu<1$, put $\delta=1-\mu$, and let $(Z_j)_{j\ge0}$ be a
$\PGW(\mu)$ process started from one particle. Set
\[
        T:=\sum_{j\ge0}Z_j,
        \qquad
        A:=\sum_{j\ge1} jZ_j,
\]
where $A$ is the total depth of the tree, equivalently the area under its
generation-size profile.
There is an absolute constant $C<\infty$ such that, uniformly in $\mu$,
\[
        \E T=\delta^{-1},\qquad
        \Var(T)\le C\delta^{-3},\qquad
        \E[(T-\E T)^4]\le C\delta^{-7},
\]
and
\[
        \E A\le C\delta^{-2},\qquad \E A^2\le C\delta^{-5}.
\]
Accordingly, if a $\PGW(\mu)$ tree is attached to a core vertex $z$, and
$W_z$ is the degree mass of the resulting rooted component as in
\eqref{eq:Wnu}, then, conditionally on the core,
\[
\begin{aligned}
        \Egr[W_z\mid\Core]&\le C(\delta^{-1}+d_{\Core}(z)),\qquad
        \Vargr(W_z\mid\Core)\le C\delta^{-3},\\
        \Egr[(W_z-\Egr W_z)^4\mid\Core]&\le C\delta^{-7}.
\end{aligned}
\]
\end{lemma}

\begin{proof}
The total progeny $T$ has the Borel distribution. If
$F(s)=\E[s^T]$ is its probability generating function, the branching property
gives
\[
        F(s)=s\exp\bigl(\mu(F(s)-1)\bigr).
\]
Differentiating this identity at $s=1$ gives
\[
        \E T=\delta^{-1},\qquad
        \E T^2=O(\delta^{-3}),\qquad
        \E T^4=O(\delta^{-7}).
\]
The $j$th derivative is obtained recursively from lower-order derivatives; at order
$j\le4$, the largest singular power is $(1-\mu)^{-(2j-1)}$. These estimates
imply the displayed variance and centered fourth-moment bounds.

For the area $A$, use the branching recursion at the root. If $\xi$ is the
number of children of the root and $(T_i,A_i)$ are independent copies of
$(T,A)$, then
\[
        A=\sum_{i=1}^\xi (T_i+A_i),\qquad \xi\sim\operatorname{Poisson}(\mu).
\]
Taking expectations gives $\E A=\mu(\E T+\E A)$, hence
$\E A=\mu\delta^{-2}$. For the second moment we use the elementary
compound-Poisson formula
\[
  \E\left(\sum_{i=1}^{\xi}B_i\right)^2
  =\mu\E B_1^2+\mu^2(\E B_1)^2,
  \qquad \xi\sim\Poi(\mu),
\]
for independent copies $B_i$. With $B_i=T_i+A_i$ and
$x=(\E A^2)^{1/2}$, Cauchy's inequality gives
\[
        x^2\le \mu\E(T+A)^2+\mu^2(\E T+\E A)^2
        \le C\delta^{-3}+C\delta^{-3/2}x+\mu x^2+C\delta^{-4}.
\]
Since $1-\mu=\delta$, this implies $x^2=O(\delta^{-5})$. For a tree attached at
a core vertex $z$, the root contributes its core degree and the attached tree
edges contribute $2(T-1)$ to the degree mass. The conditional moment estimates
for $W_z$ follow.
\end{proof}

The next three lemmas record the DKLP and configuration-model inputs used later.
They are separated so that subsequent references can point to the size, kernel,
or tree estimate needed at that moment.

\begin{lemma}[Slightly supercritical size and mixing inputs]\label{lem:standard}
With high probability the following estimates hold, with constants independent
of $n$ and $\eps$. Writing $N=|V(\Kernel)|$ and $N_c=|V(\Core)|$,
\begin{equation}\label{eq:standard-sizes}
        N=\Theta(\eps^3n),\qquad N_c=\Theta(N/\eps),\qquad
        |E(H)|=\Theta(\eps n).
\end{equation}
Equivalently, there are absolute constants $0<c_{\rm core}<C_{\rm core}<\infty$
such that, with high probability,
\begin{equation}\label{eq:core-size-explicit}
        c_{\rm core}\eps^2n\le N_c\le C_{\rm core}\eps^2n .
\end{equation}
The continuous-time mixing time satisfies
\begin{equation}\label{eq:mixing}
        \tmix(H)=O\bigl(\eps^{-3}\log^2(\eps^3n)\bigr)=o(n).
\end{equation}
\end{lemma}

\begin{proof}
The size estimates and the distributional description of the kernel, stretches,
and attached trees are part of the DKLP construction \cite[Theorem~2]{dklp}.
The proof of \cite[Theorem~1, Section~4.2]{dlp} is carried out on the contiguous
graph $\widetilde{\mathcal C}_1$ and gives lazy-walk mixing time
$O(\eps^{-3}\log^2(\eps^3n))$ there. To convert to our continuous-time
normalization, let $P$ be the non-lazy jump kernel and
$P_{\rm lazy}=(I+P)/2$. Then
\[
        e^{t(P-I)}=e^{2t(P_{\rm lazy}-I)}.
\]
Thus the continuous-time chain is the lazy chain observed after a Poisson number
of steps with mean $2t$. A Chernoff bound for that Poisson variable changes the
mixing estimate only by an absolute multiplicative factor. If
$a_n=\eps^3n\to\infty$, then
$\eps^{-3}\log^2(\eps^3n)/n=\log^2(a_n)/a_n\to0$, proving \eqref{eq:mixing}.
\end{proof}

\begin{lemma}[Uniform expansion for conditioned-Poisson kernels]
\label{lem:uniform-conditioned-poisson-expansion}
There is an absolute constant $\gamma_*>0$ with the following property. Let
$\Lambda=\Lambda_N>0$ be arbitrary. Let $D_1,\ldots,D_N$ be independent with
law $\Poi(\Lambda)$ conditioned on being at least three, condition further on
$\sum_iD_i$ being even, and let $\Kernel_N$ be the corresponding
configuration-model multigraph. If $P_{\Kernel_N}$ is the simple-random-walk
transition kernel on $\Kernel_N$, then
\[
        \Pbb\bigl(\lambda_2(I-P_{\Kernel_N})\ge\gamma_*\bigr)\longrightarrow1
\]
as $N\to\infty$, uniformly in the choice of the parameter sequence
$\Lambda_N>0$. Consequently, on the same event,
\begin{equation}\label{eq:uniform-spectral-resistance}
        \Reff^{\Kernel_N}(x,y)
        \le \gamma_*^{-1}\left(\frac1{d_{\Kernel_N}(x)}+
             \frac1{d_{\Kernel_N}(y)}\right)
        \qquad (x,y\in V(\Kernel_N)).
\end{equation}
In particular, let
\[
        \bar d_N:=\frac1N\sum_xd_{\Kernel_N}(x)
\]
and let $\widehat\pi$ be the degree-stationary law of simple random walk on the
kernel, namely
\[
        \widehat\pi(x)
        =\frac{d_{\Kernel_N}(x)}{\sum_zd_{\Kernel_N}(z)}
        =\frac{d_{\Kernel_N}(x)}{N\bar d_N}.
\]
Then
\begin{equation}\label{eq:uniform-kernel-avg-res}
        \sum_{x,y}\widehat\pi(x)\widehat\pi(y)\Reff^{\Kernel_N}(x,y)
        \le \frac{2}{\gamma_*\bar d_N}.
\end{equation}
\end{lemma}

\begin{proof}
Let $D_\Lambda$ have the law of a $\Poi(\Lambda)$ variable conditioned to be
at least three, and let
\[
        \mathcal P_N:=\left\{\sum_{i=1}^ND_i\ \text{is even}\right\}.
\]
We prove a uniform conductance bound, splitting at
$\Lambda=N^{1/100}$.

Suppose first that $0<\Lambda\le N^{1/100}$. We claim that, under the
conditioning $\mathcal P_N$,
\begin{equation}\label{eq:conditioned-poisson-max-degree}
        \max_{1\le i\le N}D_i\le N^{1/50}
\end{equation}
with probability tending to one, uniformly in this range. For
$1\le\Lambda\le N^{1/100}$, both parities of $D_\Lambda$ have probabilities
bounded away from zero uniformly in $\Lambda$. Hence
$\Pbb(\mathcal P_N)\ge c>0$, while a Poisson Chernoff bound, with
$t=N^{1/50}$, gives
\[
        \Pbb(D_\Lambda\ge t)
        \le C\left(\frac{e\Lambda}{t}\right)^t=o(N^{-1}).
\]
For $0<\Lambda<1$, write $D_\Lambda=3+Y_\Lambda$. Directly from the
conditioned-Poisson probabilities,
\[
        c\Lambda\le\Pbb(Y_\Lambda\ \text{is odd})\le C\Lambda,
        \qquad
        \Pbb(D_\Lambda\ge t)\le \frac{C\Lambda^{t-3}}{t!},
        \quad t\ge4.
\]
Since $\mathcal P_N$ prescribes one of the two parities of
$\sum_iY_i$, the parity formula
\[
 \Pbb\left(\sum_iY_i\equiv b\pmod2\right)
 =\frac{1+(-1)^b(\E(-1)^{Y_\Lambda})^N}{2}
\]
shows that the prescribed parity has probability at least
$c\min\{1,N\Lambda\}$; the more likely parity has probability at least
$1/2$. Consequently,
\[
 \Pbb\left(\max_iD_i\ge t\,\middle|\,\mathcal P_N\right)
 \le
 \frac{CN\Lambda^{t-3}/t!}{\min\{1,N\Lambda\}}=o(1)
\]
uniformly for $0<\Lambda<1$. This proves
\eqref{eq:conditioned-poisson-max-degree}.

Conditional on any degree sequence in
\eqref{eq:conditioned-poisson-max-degree}, the uniform form of
Benjamini--Kozma--Wormald \cite[Lemma~5.3]{bkw} applies: it covers every even
degree sequence satisfying
\[
        3\le\min_i d_i\le\max_i d_i\le N^{0.02}.
\]
It follows that the degree conductance of $\Kernel_N$ is at least an absolute
constant $h_0>0$ with probability tending to one, uniformly for
$\Lambda\le N^{1/100}$.

Suppose now that $\Lambda>N^{1/100}$. Poisson Chernoff bounds give
\begin{equation}\label{eq:large-lambda-degree-concentration}
        \frac{\Lambda}{2}\le D_i\le2\Lambda
        \qquad(1\le i\le N)
\end{equation}
with probability $1-o(1)$, uniformly in this range. Conditioning first on
$D_i\ge3$ and then on $\mathcal P_N$ does not change this conclusion:
$\Pbb(\Poi(\Lambda)<3)=e^{-\Omega(\Lambda)}$ and
$|\E(-1)^{D_\Lambda}|=e^{-\Omega(\Lambda)}$, so
$\Pbb(\mathcal P_N)=1/2+o(1)$.

Fix an even degree sequence satisfying
\eqref{eq:large-lambda-degree-concentration}, write
$\ell=\sum_i d_i$, and expose a uniform pairing of its $\ell$ half-edges.
For $S\subseteq[N]$, put $s=d(S):=\sum_{i\in S}d_i$ and let $B(S)$ be the
number of paired edges with exactly one endpoint in $S$. If $s\le\ell/2$,
then
\[
        \E[B(S)\mid(d_i)]
        =\frac{s(\ell-s)}{\ell-1}\ge\frac{s}{2}.
\]
Exposing successively the partners of half-edges incident to $S$ gives a
Doob martingale with at most $s$ steps and increments bounded by two. Indeed,
changing one revealed partner by a two-pair switching changes the final value
of $B(S)$ by at most two. Azuma's inequality therefore gives an absolute
$c_0>0$ such that
\begin{equation}\label{eq:pairing-boundary-tail}
 \Pbb\left(B(S)\le\frac{s}{4}\,\middle|\,(d_i)\right)
 \le e^{-c_0s}.
\end{equation}
Since $s\ge(\Lambda/2)|S|$, a union bound yields
\[
 \begin{aligned}
 &\Pbb\left(\exists S:\ d(S)\le\ell/2,\
                   B(S)<\frac14d(S)\,\middle|\,(d_i)\right)\\
 &\hspace{2em}\le \sum_{k=1}^N\binom Nk e^{-c_0\Lambda k/2}
 \le \exp\{Ne^{-c_0\Lambda/2}\}-1=o(1).
 \end{aligned}
\]
Thus the degree conductance is at least $1/4$ uniformly in the second range.
Combining the two ranges and using Cheeger's inequality gives
\[
        \lambda_2(I-P_{\Kernel_N})\ge
        \gamma_*:=\frac12\min\{h_0,1/4\}^2
\]
with the asserted uniform probability.

On this event, for every function $f$ and its degree-weighted mean $\bar f_d$,
\[
 \sum_{uv\in E(\Kernel_N)}(f(u)-f(v))^2
 \ge\gamma_*\sum_zd_{\Kernel_N}(z)(f(z)-\bar f_d)^2.
\]
Since
\[
 (f(x)-f(y))^2
 \le\left(\frac1{d_{\Kernel_N}(x)}+
           \frac1{d_{\Kernel_N}(y)}\right)
      \sum_zd_{\Kernel_N}(z)(f(z)-\bar f_d)^2,
\]
the Dirichlet principle gives
\eqref{eq:uniform-spectral-resistance}. Finally, averaging this bound under
$\widehat\pi\otimes\widehat\pi$ gives
\[
 \sum_{x,y}\widehat\pi(x)\widehat\pi(y)\Reff^{\Kernel_N}(x,y)
 \le\frac{2}{\gamma_*}\sum_x\frac{\widehat\pi(x)}
                                      {d_{\Kernel_N}(x)}
 =\frac{2}{\gamma_*\bar d_N},
\]
which is \eqref{eq:uniform-kernel-avg-res}.
\end{proof}
\begin{lemma}[Kernel expansion and stretch inputs]\label{lem:kernel-inputs}
With high probability, for every fixed $r$,
\begin{equation}\label{eq:degree-moments}
        \sum_{a\in V(\Kernel)}d_{\Kernel}(a)^r=O_r(N),
\end{equation}
and consequently $\sum_{z\in V(\Core)}d_{\Core}(z)^r=O_r(N_c)$ for every fixed
$r$. Let $P_{\Kernel}$ be the discrete-time simple-random-walk kernel on
$\Kernel$. The smallest nonzero eigenvalue of $I-P_{\Kernel}$, acting on
$L^2$ of the degree-stationary law, is bounded below by an absolute constant.
In particular, if $\Delta_{\Kernel}$ is the unnormalized graph Laplacian, then
\[
        f^{\mathsf T}\Delta_{\Kernel}f
        =\sum_{ab\in E(\Kernel)}(f(a)-f(b))^2
        \ge c\sum_{a\in V(\Kernel)}f(a)^2
\]
whenever $\sum_a f(a)=0$. Thus every pair of kernel vertices has
effective resistance $O(1)$.

If $L_e$ denotes the stretch of a kernel edge, then for every fixed $r$,
\begin{equation}\label{eq:stretch-moments}
        \sum_{e\in E(\Kernel)}L_e^r=O_r(N\eps^{-r}).
\end{equation}
In addition, for every fixed $r$ and every kernel-measurable choice of edge
weights $q_e$, fixed after conditioning on the kernel and independent of the
stretch lengths, with $0\le q_e\le1$ and $\sum_e q_e=O(N)$,
\[
        \sum_e L_e^r q_e=O_r(N\eps^{-r})
\]
with high probability over the stretches; the same remains true after
multiplying $q_e$ by any fixed-degree polynomial in the endpoint degrees.
\end{lemma}

\begin{proof}
The kernel degree-moment bounds follow from Poisson tails and concentration of
the number of vertices of degree at least three in the DKLP construction. The
core-degree bounds follow because non-kernel core vertices have degree two.

The expansion input follows from
Lemma~\ref{lem:uniform-conditioned-poisson-expansion}, applied conditionally on
the DKLP parameter and the retained kernel size $(\Lambda,N)$. The kernel degree
law is a Poisson law conditioned to
be at least three, and the high-probability range of the random $\Lambda$ is
covered by the uniform statement of that lemma. To obtain the displayed
unnormalized inequality, let $\bar f_d$ be the degree-weighted mean of $f$. The
normalized gap gives
\[
 f^{\mathsf T}\Delta_{\Kernel}f
 \ge c\sum_a d_{\Kernel}(a)(f(a)-\bar f_d)^2
 \ge 3c\sum_a(f(a)-\bar f_d)^2.
\]
If $\sum_a f(a)=0$, then
$\sum_a(f(a)-\bar f_d)^2=\sum_a f(a)^2+N\bar f_d^{\,2}$, which is at least
$\sum_a f(a)^2$. Hence the second eigenvalue of $\Delta_{\Kernel}$ is bounded below
by an absolute constant. Since
$\Reff^{\Kernel}(a,b)=(e_a-e_b)^{\mathsf T}\Delta_{\Kernel}^{\dagger}(e_a-e_b)$,
the operator bound
$\Delta_{\Kernel}^{\dagger}\preceq\lambda_2(\Delta_{\Kernel})^{-1}I$ on the mean-zero
subspace yields $\Reff^{\Kernel}(a,b)\le2/\lambda_2(\Delta_{\Kernel})=O(1)$.

The stretch-moment bounds come from independence and the geometric tails of the
stretch lengths. Conditional on the kernel,
$\Egr L_e^r=O_r(\eps^{-r})$ and $\Vargr(L_e^r)=O_r(\eps^{-2r})$. Chebyshev's
inequality and \eqref{eq:degree-moments} give \eqref{eq:stretch-moments}. The
weighted estimates use the same conditional expectation and variance calculation,
with the polynomial weights controlled by \eqref{eq:degree-moments}.
\end{proof}

\begin{lemma}[Attached-tree mass inputs]\label{lem:tree-mass-inputs}
For the tree attached to a core vertex $z$, the random mass
$W_z=\sum_{u:r(u)=z}d(u)$ satisfies, conditionally on $\Core$,
\begin{equation}\label{eq:W-moments}
        \bar W_z:=\Egr[W_z\mid\Core]\le C(\eps^{-1}+d_{\Core}(z)),
        \qquad
        \Vargr(W_z\mid\Core)=O(\eps^{-3}),
\end{equation}
and
\begin{equation}\label{eq:W-fourth}
        \Egr\bigl[(W_z-\bar W_z)^4\mid\Core\bigr]=O(\eps^{-7}).
\end{equation}
Also
\begin{equation}\label{eq:S-size}
        S=\sum_zW_z=2|E(H)|=\Theta(N_c/\eps)=\Theta(\eps n)
\end{equation}
with high probability.
\end{lemma}

\begin{proof}
Condition on the core $\Core$. The attached $\PGW(\mu)$ trees at different
core vertices remain independent. Applying Lemma~\ref{lem:pgw-moments} with
$\delta=1-\mu\asymp\eps$ gives \eqref{eq:W-moments} and
\eqref{eq:W-fourth}; the deterministic root contribution is controlled by the
core degree moments in Lemma~\ref{lem:kernel-inputs}. Summing the independent
masses $W_z$ and using $N_c\eps=\Theta(\eps^3n)\to\infty$ gives the
concentration of $S$.
\end{proof}

\begin{lemma}[Weighted core resistance]\label{lem:core-resistance}
There is an absolute constant $C_{\rm cr}<\infty$ such that, with high probability in
the DKLP model,
\begin{equation}\label{eq:core-res-input}
        \sum_{z,w\in V(\Core)}\nu(z)\nu(w)\Reff^{\Core}(z,w)
        \le C_{\rm cr}\eps^{-1}.
\end{equation}
The constant $C_{\rm cr}$ is independent of $n$ and $\eps$.
\end{lemma}

\begin{proof}
Work on the event supplied by Lemmas~\ref{lem:standard}, \ref{lem:kernel-inputs},
and~\ref{lem:tree-mass-inputs}. On this event the unnormalized kernel Laplacian
$\Delta_{\Kernel}$ satisfies
\begin{equation}\label{eq:kernel-gap-detail}
        f^{\mathsf T}\Delta_{\Kernel}f\ge \gamma_0\sum_{a\in V(\Kernel)}f(a)^2
        \quad\text{whenever }\sum_a f(a)=0,
\end{equation}
for an absolute constant $\gamma_0>0$.  Also, for $r=1,2,3,4$, there are
absolute constants $A_r$ such that
\begin{equation}\label{eq:stretch-sums-detail}
        \sum_{e\in E(\Kernel)} L_e^r\le A_r N\eps^{-r},
        \qquad N:=|V(\Kernel)|.
\end{equation}
The constants $A_r$ come from the conditional geometric moment calculation in
Lemma~\ref{lem:kernel-inputs}. Choose the nearest-kernel projection
$a(z)\in V(\Kernel)$ of a core vertex $z\in V(\Core)$, write
$\delta(z)=\dist_{\Core}(z,a(z))$, and set
$n_a:=|\{z\in V(\Core):a(z)=a\}|$. The same stretch moment bounds give absolute
constants $A_\delta,A_{\delta,2},A_n$ such that
\begin{equation}\label{eq:projection-sums-detail}
        \sum_z\delta(z)\le A_\delta \frac{N_c}{\eps},
        \qquad
        \sum_z\delta(z)^2\le A_{\delta,2}\frac{N_c}{\eps^2},
        \qquad
        \sum_a n_a^2\le A_n \frac{N}{\eps^2},
\end{equation}
where $N_c:=|V(\Core)|$. On a stretched edge of length $L$, the sum of distances
to the nearer endpoint is at most $L^2/2$ and the sum of squared distances is at
most $L^3$; summing over edges and using \eqref{eq:stretch-sums-detail} proves
the first two inequalities. For the third, $n_a\le \sum_{e\sim a}L_e$, and
Cauchy's inequality gives $n_a^2\le d_{\Kernel}(a)\sum_{e\sim a}L_e^2$; the
weighted stretch estimate from Lemma~\ref{lem:kernel-inputs}, with weight
controlled by the kernel degree moment bound, gives the displayed inequality.

Collapse every stretched path to one kernel edge with resistance equal to its
length.  Denote the resulting weighted network by $\Kernel^L$.  For ordered
kernel vertices $a,b$, let $\theta^{a,b}$ be the unit electrical current from
$a$ to $b$ in the unweighted kernel, with $\theta^{a,a}=0$.  Thomson's principle
and Rayleigh monotonicity give
\begin{equation}\label{eq:thomson-kernelL-detail}
        \Reff^{\Kernel^L}(a,b)
        \le \sum_{e\in E(\Kernel)}L_e(\theta_e^{a,b})^2 .
\end{equation}
For $a\ne b$, put $R_{ab}:=\sum_e(\theta_e^{a,b})^2=\Reff^\Kernel(a,b)$.
By \eqref{eq:kernel-gap-detail}, $R_{ab}\le 2\gamma_0^{-1}$.  Jensen's
inequality applied to the probability weights
$(\theta_e^{a,b})^2/R_{ab}$ yields, for $r=1,2,4$,
\begin{equation}\label{eq:jensen-kernelL-detail}
        \Reff^{\Kernel^L}(a,b)^r
        \le \Big(\frac{2}{\gamma_0}\Big)^{r-1}
        \sum_e L_e^r(\theta_e^{a,b})^2 .
\end{equation}
The estimate is trivial when $a=b$.

Now average the squared currents over all ordered source-target pairs. For a
fixed oriented kernel edge $e$, let $b_e$ be its signed incidence vector. The
voltage representation of currents gives
\[
        \theta_e^{a,b}=b_e^{\mathsf T}\Delta_{\Kernel}^{\dagger}(e_a-e_b).
\]
Since
\[
        \sum_{a,b}(e_a-e_b)(e_a-e_b)^{\mathsf T}=2N P_0,
\]
where $P_0$ is orthogonal projection onto the subspace orthogonal to constants,
and since $\Delta_{\Kernel}^{\dagger}$ vanishes on constants,
\begin{align}
        \sum_{a,b}(\theta_e^{a,b})^2
        &=2N b_e^{\mathsf T}(\Delta_{\Kernel}^{\dagger})^2 b_e        \notag\\
        &\le \frac{2N}{\gamma_0}\, b_e^{\mathsf T}\Delta_{\Kernel}^{\dagger} b_e
        \le \frac{2N}{\gamma_0}.                                      \label{eq:current-average-detail}
\end{align}
Here $\gamma_0$ is the absolute spectral-gap constant fixed in
\eqref{eq:kernel-gap-detail}. The last inequality uses that $b_e^{\mathsf T}\Delta_{\Kernel}^{\dagger} b_e$ is the
effective resistance between the endpoints of $e$ in the unweighted kernel, and
is therefore at most the resistance of the edge itself, namely $1$. Combining
\eqref{eq:jensen-kernelL-detail}, \eqref{eq:current-average-detail}, and
\eqref{eq:stretch-sums-detail} gives, for $r=1,2,4$,
\begin{equation}\label{eq:kernelL-moments-detail}
        \sum_{a,b}\Reff^{\Kernel^L}(a,b)^r
        \le B_r N^2\eps^{-r},
        \qquad
        B_r:=\Big(\frac{2}{\gamma_0}\Big)^r A_r .
\end{equation}

For core vertices $z,w$,
\begin{equation}\label{eq:core-to-kernel-detail}
        \Reff^\Core(z,w)
        \le \delta(z)+\Reff^{\Kernel^L}(a(z),a(w))+\delta(w).
\end{equation}
Using \eqref{eq:projection-sums-detail} and \eqref{eq:kernelL-moments-detail},
we get
\begin{align}
        R_1&:=\sum_{z,w}\Reff^\Core(z,w)
        \le C_R^{(1)} N_c^2\eps^{-1},                                  \label{eq:R1-detail}\\
        R_2&:=\sum_{z,w}\Reff^\Core(z,w)^2
        \le C_R^{(2)} N_c^2\eps^{-2},                                  \label{eq:R2-detail}
\end{align}
for absolute constants $C_R^{(1)},C_R^{(2)}$. We spell out the only non-immediate term. For
the kernel part in $R_1$,
\[
        \sum_{z,w}\Reff^{\Kernel^L}(a(z),a(w))
        =\sum_{a,b}n_a n_b\Reff^{\Kernel^L}(a,b)
        \le \Big(\sum_a n_a^2\Big)
             \Big(\sum_{a,b}\Reff^{\Kernel^L}(a,b)^2\Big)^{1/2},
\]
which is at most $A_n B_2^{1/2}N^2\eps^{-3}$. Since
$N_c=\Theta(N/\eps)$ on the present event, this is bounded by an absolute
constant times $N_c^2\eps^{-1}$. The kernel contribution to $R_2$ is the same,
using the fourth moment in \eqref{eq:kernelL-moments-detail}.

We now insert the random tree weights $W_z$. Condition on the core $\Core$ and
write
\[
        \bar W_z:=\Egr[W_z\mid \Core],\qquad \xi_z:=W_z-\bar W_z.
\]
From Lemmas~\ref{lem:pgw-moments}, \ref{lem:kernel-inputs}, and~\ref{lem:tree-mass-inputs}, there are absolute
constants $C_W,C_V,C_4$ such that, conditionally on $\Core$,
\begin{equation}\label{eq:W-constants-detail}
        \sum_z \bar W_z^2\le C_W N_c\eps^{-2},
        \qquad
        \Vargr(\xi_z\mid\Core)\le C_V\eps^{-3},
        \qquad
        \Egr[\xi_z^4\mid\Core]\le C_4\eps^{-7}.
\end{equation}
Also $S=\sum_z W_z$ satisfies $S\ge c_S N_c\eps^{-1}$ with high probability for
an absolute $c_S>0$.

Set
\[
        T:=\sum_{z,w}W_zW_w\Reff^\Core(z,w).
\]
By Cauchy's inequality and \eqref{eq:R2-detail}, the deterministic part satisfies
\begin{equation}\label{eq:det-part-detail}
        \sum_{z,w}\bar W_z\bar W_w\Reff^\Core(z,w)
        \le \Big(\sum_z\bar W_z^2\Big)R_2^{1/2}
        \le C_W \bigl(C_R^{(2)}\bigr)^{1/2}N_c^2\eps^{-3}.
\end{equation}
The linear fluctuation has conditional variance at most
\begin{align}
        \Vargr\!\left(2\sum_z\xi_z\sum_w\bar W_w\Reff^\Core(z,w)\,
        \middle|\,\Core\right)
        &\le 4C_V\eps^{-3}\sum_z\left(\sum_w\bar W_w\Reff^\Core(z,w)\right)^2
            \notag\\
        &\le 4C_V\eps^{-3}\Big(\sum_w\bar W_w^2\Big)R_2              \notag\\
        &\le 4C_V C_W C_R^{(2)} N_c^3\eps^{-7}.                \label{eq:linear-var-detail}
\end{align}
This term is $o(N_c^2\eps^{-3})$ in probability, because
$N_c\eps=\Theta(\eps^3n)\to\infty$. For the quadratic fluctuation, the diagonal
terms vanish because $\Reff^\Core(z,z)=0$, and independence gives
\begin{equation}\label{eq:quadratic-var-detail}
        \Egr\!\left[\left(\sum_{z,w}\xi_z\xi_w\Reff^\Core(z,w)\right)^2
        \middle|\Core\right]
        \le 2C_V^2\eps^{-6}R_2
        \le 2C_V^2C_R^{(2)} N_c^2\eps^{-8}.
\end{equation}
This is $o(N_c^4\eps^{-6})$ as well. Combining
\eqref{eq:det-part-detail}, \eqref{eq:linear-var-detail}, and
\eqref{eq:quadratic-var-detail}, we obtain
\[
        T\le C_T N_c^2\eps^{-3}
\]
with high probability, for an absolute constant $C_T$. Dividing by
$S^2\ge c_S^2N_c^2\eps^{-2}$ gives \eqref{eq:core-res-input} with
$C_{\rm cr}=C_T/c_S^2$.
\end{proof}

\begin{lemma}[Average target time]\label{lem:avg-target}
There is an absolute constant $C_{\rm K}<\infty$ such that, with high probability
in the DKLP model,
\begin{equation}\label{eq:avg-target-input}
        \sum_{v\in V(H)}\pi(v)\E^H_\pi\tau_v\le C_{\rm K} n .
\end{equation}
\end{lemma}

\begin{proof}
By Lemma~\ref{lem:ct-commute},
\begin{equation}\label{eq:avg-target-res}
        \sum_v\pi(v)\E^H_\pi\tau_v
        =|E(H)|\sum_{u,v}\pi(u)\pi(v)\Reff^H(u,v).
\end{equation}
For $u,v\in V(H)$, follow the path from $u$ to its core root, then through the
core, and then from the root of $v$ to $v$. This gives
\begin{equation}\label{eq:res-decomp-target-detail}
        \Reff^H(u,v)
        \le h(u)+\Reff^\Core(r(u),r(v))+h(v).
\end{equation}
The middle term has $\pi\otimes\pi$ average at most $C_{\rm cr}\eps^{-1}$ by
Lemma~\ref{lem:core-resistance}.

It remains to bound the tree-depth term. For one attached $\PGW(\mu)$ tree, let
$Z_j$ be the number of vertices at height $j$ and set $A:=\sum_{j\ge1}jZ_j$.
Lemma~\ref{lem:pgw-moments} gives $\E A\le C_A\eps^{-2}$ and
$\E A^2\le C_A'\eps^{-5}$ with absolute constants. For a tree rooted at $z$,
the degree-weighted height satisfies
\begin{equation}\label{eq:tree-height-degree-detail}
        \sum_{u:r(u)=z}d_H(u)h(u)\le 2A_z,
\end{equation}
where $A_z$ is the corresponding copy of $A$. To check
\eqref{eq:tree-height-degree-detail}, write $d_H(u)$ for a non-root tree vertex
as $1+$ its number of children; the first contribution is $A_z$, and the
contribution from children is at most $A_z$ after assigning each child edge to
its child.

Summing the independent variables $A_z$ over the $N_c$ core vertices and using
$N_c\eps=\Theta(\eps^3n)\to\infty$ gives, with high probability,
\[
        \sum_z A_z\le C_A'' N_c\eps^{-2}
\]
for an absolute $C_A''$. Since $S=2|E(H)|\ge c_S N_c\eps^{-1}$,
\begin{equation}\label{eq:stationary-height-detail}
        \sum_u\pi(u)h(u)
        =\frac{1}{S}\sum_z\sum_{u:r(u)=z}d_H(u)h(u)
        \le \frac{2C_A''}{c_S}\eps^{-1}.
\end{equation}
Averaging \eqref{eq:res-decomp-target-detail} and using
\eqref{eq:stationary-height-detail} and Lemma~\ref{lem:core-resistance} yields
\[
        \sum_{u,v}\pi(u)\pi(v)\Reff^H(u,v)\le C_R\eps^{-1}
\]
for an absolute $C_R$. Lemma~\ref{lem:standard} gives
$|E(H)|\le C_E\eps n$ with high probability. Substituting in
\eqref{eq:avg-target-res} proves \eqref{eq:avg-target-input} with
$C_{\rm K}:=C_EC_R$.
\end{proof}

\begin{lemma}[Kernel balls at the stretch scale]\label{lem:kernel-ball-exploration}
Let $k=\lfloor\eps^{-1}\rfloor$ and fix $A<\infty$. For a kernel vertex $a$,
let $K_E^{(A)}(a)$ be the number of kernel edges incident to the weighted
kernel ball of radius $Ak$ around $a$, where the weight of a kernel edge is its
stretch length $L_e$. Thus
\[
 K_E^{(A)}(a)
 :=\bigl|\{e\in E(\Kernel):
       e\text{ has an endpoint }b\text{ with }
       \dist_{\Kernel^L}(a,b)\le Ak\}\bigr|.
\]
Let
\[
        K_D^{(A)}(a):=\sum_{b:\,\dist_{\Kernel^L}(a,b)\le Ak}d_\Kernel(b).
\]
For each fixed integer $r\ge1$ there is a constant $C_{A,r}<\infty$, depending
only on $A$ and $r$, such that, with high probability,
\begin{equation}\label{eq:kernel-ball-moments-general}
        \sum_{a\in V(\Kernel)}K_E^{(A)}(a)^r\le C_{A,r}N,
        \qquad
        \sum_{a\in V(\Kernel)}K_D^{(A)}(a)^r\le C_{A,r}N .
\end{equation}
\end{lemma}

\begin{proof}
Condition on a kernel degree sequence satisfying the moment estimates in
Lemma~\ref{lem:kernel-inputs}. Let $\delta_\mu=1-\mu$, so $\delta_\mu\asymp\eps$.
Give each kernel edge the rescaled mark $\xi=\delta_\mu L$, where
$L\sim\operatorname{Geom}(1-\mu)$ on $\{1,2,\ldots\}$. For every $\theta>0$,
\begin{equation}\label{eq:laplace-mark-detail}
        \sup_{0<\delta_\mu<1}\E e^{-\theta\xi}
        =\sup_{0<\delta_\mu<1}
        \frac{\delta_\mu e^{-\theta\delta_\mu}}{1-(1-\delta_\mu)e^{-\theta\delta_\mu}}
        =:q_\theta<1,
\end{equation}
and $q_\theta\downarrow0$ as $\theta\to\infty$.

We dominate the exploration by an age-dependent branching process. Define the
empirical kernel degree law and its forward size-biased version by
\begin{equation}\label{eq:pn-hatpn}
        p_n(j):=\frac1N\bigl|\{a\in V(\Kernel):d_{\Kernel}(a)=j\}\bigr|,
        \qquad j\ge3,
\end{equation}
\begin{equation}\label{eq:hatpn}
        \widehat p_n(m):=
        \frac{(m+1)p_n(m+1)}{\sum_{j\ge3}j p_n(j)},
        \qquad m\ge2 .
\end{equation}
Thus $p_n$ is the law of the degree of a uniform root, while $\widehat p_n$ is
the law of the number of fresh forward half-edges at a new non-root vertex
reached by following a uniformly chosen unmatched half-edge. Let $D_n\sim p_n$
and $\widehat D_n\sim\widehat p_n$. Because the degree sequence has bounded
moments of every fixed order, $D_n$ and $\widehat D_n$ have bounded moments of
every fixed order, uniformly in $n$.

We use the following exploration bound. For the fixed moment orders needed
below, work on the high-probability degree-sequence event from
Lemma~\ref{lem:kernel-inputs} on which, with deterministic constants,
\begin{equation}\label{eq:degree-event-for-exploration}
        3N\le \ell_N:=\sum_a d_{\Kernel}(a)\le C N,
        \qquad
        \sum_a d_{\Kernel}(a)^m\le C_mN
\end{equation}
for every $m$ up to the required finite order. For a desired exponent $q$ in
the claim below, we take these bounds through order $q+2$ and use moment order
$s=q+1$ in the stopping estimate.

\begin{claim}\label{claim:exploration-domination}
Conditional on a degree sequence satisfying
\eqref{eq:degree-event-for-exploration}, for every fixed $A<\infty$ and every
integer $q\ge1$,
\begin{equation}\label{eq:rooted-ball-uniform-moments}
 \sup_n\E\left[K_E^{(A)}(U)^q+K_D^{(A)}(U)^q\right]<\infty,
\end{equation}
where $U$ is a uniform kernel vertex and the expectation is over $U$, the
uniform pairing, and the stretch lengths. The same assertion holds for the
number of half-edges revealed while exploring the weighted ball.
\end{claim}

\begin{proof}[Proof of the claim]
Let $\widetilde D_n$ be the nonnegative integer-valued variable defined by
\begin{equation}\label{eq:dominating-forward-degree}
 \Pbb(\widetilde D_n\ge m)
 :=\min\{1,2\Pbb(\widehat D_n\ge m)\},\qquad m\ge0.
\end{equation}
By \eqref{eq:degree-event-for-exploration}, for every fixed $q$,
\[
 \sup_n\E\bigl[D_n^q+\widehat D_n^q+\widetilde D_n^q\bigr]<\infty.
\]
Indeed,
\[
 \E\widehat D_n^q
 \le \frac{\sum_a d_{\Kernel}(a)^{q+1}}{\ell_N},
 \qquad
 \E\widetilde D_n^q\le2\E\widehat D_n^q,
\]
where the last inequality follows by summing the tail probabilities of the
nonnegative integer-valued variables.

Run Dijkstra exploration from $U$, but stop before initiating any pairing for
which fewer than $\ell_N/2$ possible partners remain. Before that stopping
time, the partner of every paired half-edge is uniform in a pool of size at
least $\ell_N/2$. Consequently, even
if that partner belongs to a vertex encountered earlier, its full degree
satisfies
\[
 \Pbb\bigl(d_{\Kernel}(\text{partner})-1\ge m
       \mid\mathcal F\bigr)
 \le
 \frac{\sum_{a:\,d_{\Kernel}(a)-1\ge m}d_{\Kernel}(a)}{\ell_N/2}
 \le2\Pbb(\widehat D_n\ge m),
\]
where $\mathcal F$ is the exploration history before the pairing.

We use an arrival-ticket version of Dijkstra's algorithm. Whenever a half-edge
at a settled vertex is paired, create a ticket carrying the partner vertex and
the proposed arrival time, whether or not that vertex has appeared before.
Process tickets in increasing order of arrival time, with a fixed rule for
ties. The first ticket reaching a vertex settles it; every later ticket for the
same vertex is declared dummy.
Thus a vertex that first receives a long tentative route and later a shorter
one is settled by the shorter ticket.

Couple these tickets to an age-dependent branching process whose root has
$D_n$ children, whose other individuals have $\widetilde D_n$ children, and
whose edge ages are independent copies of
\[
        \xi_\delta:=\delta L,
        \qquad L\sim\operatorname{Geom}(\delta),
        \qquad \delta=1-\mu.
\]
For each non-root ticket, the preceding conditional tail bound and a sequential
quantile coupling provide a fresh copy of $\widetilde D_n$ that dominates
$d_{\Kernel}(v)-1$ for its partner $v$. These dominating variables may be taken
independent across tickets and independent of the edge ages. If the ticket
settles $v$, the number of still-unmatched half-edges at $v$ is at most
$d_{\Kernel}(v)-1$, so its outgoing tickets can be matched injectively to the
branching-process children, using the same edge ages. If the ticket is dummy,
it creates no graph tickets, while all of its branching-process descendants
are retained and merely overcount.

Induction over the ordered ticket times now shows that every graph ticket is
represented by a distinct branching-process particle born at the same time.
In particular, every vertex settled by time $T$ is represented by a particle
born by time $T$, and its degree is bounded by that particle's full degree.
The number of kernel edges incident to the settled set is at most its total
degree volume, so both $K_E^{(A)}$ and $K_D^{(A)}$ are dominated by the
branching-process degree volume born by time $T$. The number of revealed
half-edges is at most twice that volume. This induction is the point at which
retaining dummy collision tickets is essential.

For every $\theta>0$,
\begin{equation}\label{eq:uniform-age-laplace}
 q_\theta:=\sup_{0<\delta<1}\E e^{-\theta\xi_\delta}
 =\sup_{0<\delta<1}
   \frac{\delta e^{-\theta\delta}}
        {1-(1-\delta)e^{-\theta\delta}}<1,
 \qquad q_\theta\downarrow0
\end{equation}
as $\theta\to\infty$. Let $N_j$ be the size of generation $j$ in the
unmarked branching process. The uniform offspring moments and the usual
branching recursion give, for every fixed integer $q\ge1$,
\begin{equation}\label{eq:generation-moments}
        \E N_j^q\le C_q^j,\qquad j\ge0.
\end{equation}
Let $R_0:=D_n$, and for $j\ge1$ let $R_j$ be the sum of the full degrees
$1+\widetilde D_n$ of the generation-$j$ individuals. The corresponding
estimate is
\[
        \E R_j^q\le C_q^{j+1},\qquad j\ge0.
\]

Let $Z_{T,j}$ be the number of generation-$j$ individuals born by time $T$.
Conditional on the genealogical tree, each root-to-generation-$j$ path has the
same age-sum distribution as $S_j=\xi_1+\cdots+\xi_j$. Since
$Z_{T,j}^q\le N_j^{q-1}Z_{T,j}$,
\[
 \E Z_{T,j}^q
 \le \E N_j^q\,\Pbb(S_j\le T)
 \le e^{\theta T}(C_qq_\theta)^j.
\]
Choose $\theta$ so that $C_qq_\theta<1$. Minkowski's inequality then yields
\begin{equation}\label{eq:age-dependent-moment-detail}
 \sup_n\sup_{0<\delta<1}\E Z_T^q<\infty.
\end{equation}
Replacing $N_j$ by $R_j$ in the same calculation gives the corresponding
bound for the degree volume $Q_T$ of individuals born by time $T$. Indeed, if
$Q_{T,j}$ is the generation-$j$ contribution, then
\[
 Q_{T,j}^q\le R_j^{q-1}
 \sum_{|u|=j}d(u)\one_{\{S(u)\le T\}},
\]
and conditioning on the genealogy gives the same bound.

Since $(1-\mu)k=O(1)$, the weighted radius $Ak$ becomes, after rescaling,
a time horizon $T_A<\infty$ depending only on $A$. Thus the stopped kernel
ball and all half-edges revealed in constructing it are dominated by the
preceding branching-process quantities.

It remains to remove the stopping. If the stopping rule is reached, at least
half of the original half-edge pool has been removed, so the dominating
branching process has degree volume at least
$\ell_N/5\ge3N/5$ for all large $N$, because each initiated pairing removes
at most its active half-edge and its partner. Using the moment bound at order
$s=q+1$ gives
\[
 \Pbb(\text{the stopping time is reached})
 \le \Pbb(Q_{T_A}\ge3N/5)\le C_{A,q}N^{-(q+1)}.
\]
On \eqref{eq:degree-event-for-exploration}, deterministically
$K_E^{(A)}(U)\le\ell_N/2\le CN$ and
$K_D^{(A)}(U)\le\ell_N\le CN$. Therefore
\[
 \E K_E^{(A)}(U)^q+\E K_D^{(A)}(U)^q
 \le C_{A,q}+(CN)^q C_{A,q}N^{-(q+1)}\le C'_{A,q}.
\]
The identical argument bounds the revealed half-edge count. This proves
\eqref{eq:rooted-ball-uniform-moments}.
\end{proof}

In particular, applying the claim with moment order $2r+2$ gives
\begin{equation}\label{eq:root-moment-detail}
        \sup_n\E K_E^{(A)}(U)^{2r+2}<\infty,
        \qquad
        \sup_n\E K_D^{(A)}(U)^{2r+2}<\infty,
\end{equation}
and the corresponding families, as well as the revealed half-edge counts, are
uniformly integrable.
We pass from the rooted estimate to an average over all roots. Put
$Y_a=K_E^{(A)}(a)^r$ or $Y_a=K_D^{(A)}(a)^r$. Let $\mathcal H_a$ be the
total discovered degree volume plus the paired boundary half-edges revealed by
the exploration from $a$. For a budget $M$, stop before accepting a degree
block that would make this quantity exceed $M$, and let
\[
        Y_a^{(M)}:=Y_a\one_{\{\mathcal H_a\le M\}}.
\]
Thus $Y_a^{(M)}\le M^r$. Determining it reveals at most $M+O(1)$ half-edges
and touches at most $M+O(1)$ vertex degree blocks; the additive term accounts
for the rejected boundary block and the final paired partner.

For fixed $M$, consider two such explorations from independent uniform roots.
If $\mathcal V_1$ is the set of degree blocks touched by the first, then
the size-biased partner bound and
$\sum_vd_{\Kernel}(v)^2=O(N)$ give
\[
        \E\left[\sum_{v\in\mathcal V_1}d_{\Kernel}(v)\right]=O_M(1).
\]
The second root lies in $\mathcal V_1$ with probability $O_M(N^{-1})$, and the
second exploration makes only $O_M(1)$ partner selections. Since
$\ell_N\asymp N$, a union bound therefore gives
\[
        \Pbb(\text{the two explorations touch a common degree block})
        =O_M(N^{-1}).
\]
This count includes intersections through different half-edges of the same
vertex, not only literal reuse of an exposed half-edge. On the complementary
event, removing the $O_M(1)$ half-edges used by the first exploration changes
the law of the second by total variation $O_M(N^{-1})$. Consequently,
\[
 \Var\left(\frac1N\sum_aY_a^{(M)}\right)=O_M(N^{-1}).
\]
Claim~\ref{claim:exploration-domination}, used with moment order $2r+2$, shows
that $\{Y_U\}_n$ and $\{\mathcal H_U\}_n$ have uniformly bounded moments of
order strictly larger than one. H\"older's inequality, followed by
Markov's inequality for $\mathcal H_U$, therefore gives
\[
        \sup_n\E\bigl[Y_U\one_{\{\mathcal H_U>M\}}\bigr]
        \longrightarrow0
        \qquad(M\to\infty).
\]
We may therefore choose a
deterministic sequence $M_N\uparrow\infty$ slowly enough that
\[
 \Var\left(\frac1N\sum_aY_a^{(M_N)}\right)=o(1),
 \qquad
 \E\bigl|Y_U-Y_U^{(M_N)}\bigr|=o(1).
\]
The expectations of the truncated averages are bounded uniformly, so the first
relation and Chebyshev's inequality bound each truncated average by a
deterministic constant with probability tending to one. By exchangeability of
the uniform root and Markov's inequality,
\[
 \Pbb\left(\frac1N\sum_a|Y_a-Y_a^{(M_N)}|>1\right)
 \le \E|Y_U-Y_U^{(M_N)}|=o(1).
\]
Thus $N^{-1}\sum_aY_a$ is bounded by a deterministic constant with high
probability. Applying this argument to both choices of $Y_a$ proves
\eqref{eq:kernel-ball-moments-general}.
\end{proof}

\begin{lemma}[Weighted core balls]\label{lem:core-balls}
Let $k=\lfloor\eps^{-1}\rfloor$.  There are absolute constants
$C_{\rm ball},C_{\rm near}<\infty$ such that, with high probability,
\begin{equation}\label{eq:BE-input}
        \sum_{z\in V(\Core)}\nu(z)B_E(z,k+2)
        \le C_{\rm ball} k,
\end{equation}
and
\begin{equation}\label{eq:near-input}
        \sum_{z,w\in V(\Core)}\nu(z)\nu(w)
        \one_{\{\dist_{\Core}(z,w)\le k\}}
        \le \frac{C_{\rm near}}{\eps^3n}=o(1).
\end{equation}
\end{lemma}

\begin{proof}
Let $N=|V(\Kernel)|$ and $N_c=|V(\Core)|$. By
Lemma~\ref{lem:kernel-ball-exploration} with $A=4$, and by the stretch moment
estimates in Lemma~\ref{lem:kernel-inputs}, work on an event on which
\begin{equation}\label{eq:KEKD-fourth-detail}
        \sum_a K_E(a)^4\le C_E^{(4)}N,
        \qquad
        \sum_a K_D(a)^4\le C_D^{(4)}N,
\end{equation}
where $K_E,K_D$ denote the radius-$4k$ weighted kernel quantities. Also, with
$c_a:=\sum_{e\sim a}L_e$,
\begin{equation}\label{eq:c-sums-detail}
        \sum_a c_a\le C_c N_c,
        \qquad
        \sum_a c_a^2\le C_c' N\eps^{-2}.
\end{equation}

If $z$ lies on the stretched image of a kernel edge with endpoint kernel
vertices $a,b$ (with $a=b$ for a stretched self-loop), then every core edge
incident to $B_\Core(z,k+2)$ is contained in the part of that stretched edge
within distance $k+2$ of $z$, together with the first $k+2$ core edges along
stretched edges reached from endpoint kernel vertices inside the weighted kernel
ball of radius $4k$. Hence
\begin{equation}\label{eq:BE-local-detail}
        B_E(z,k+2)
        \le C_0 k\bigl(1+K_E(a)+K_E(b)\bigr),
\end{equation}
for an absolute $C_0$. Similarly,
\begin{equation}\label{eq:BD-local-detail}
        B_D(z,k)
        \le C_0 k\bigl(1+K_D(a)+K_D(b)\bigr).
\end{equation}
Squaring \eqref{eq:BE-local-detail}, summing over all core vertices, and using
Cauchy's inequality with \eqref{eq:KEKD-fourth-detail} and
\eqref{eq:c-sums-detail} gives
\begin{equation}\label{eq:BE-square-detail}
        \sum_z B_E(z,k+2)^2\le C_{BE} N_c k^2.
\end{equation}
The same argument gives
\begin{equation}\label{eq:BD-square-detail}
        \sum_z B_D(z,k)^2\le C_{BD} N_c k^2.
\end{equation}
For instance,
\[
        \sum_z B_E(z,k+2)^2
        \le C k^2\sum_a c_a(1+K_E(a)^2)
        \le C k^2\left(\sum_a c_a+\big(\sum_a c_a^2\big)^{1/2}
        \big(\sum_aK_E(a)^4\big)^{1/2}\right),
\]
which is bounded by an absolute constant times $N_ck^2$ because
$N_c=\Theta(N/\eps)$.

We prove \eqref{eq:BE-input}. Conditional on $\Core$, write
$W_z=\bar W_z+\xi_z$ as in the proof of Lemma~\ref{lem:core-resistance}. Then
\[
        \sum_z\bar W_z B_E(z,k+2)
        \le \Big(\sum_z\bar W_z^2\Big)^{1/2}
             \Big(\sum_zB_E(z,k+2)^2\Big)^{1/2}
        \le C N_c\eps^{-2},
\]
using \eqref{eq:W-constants-detail} and \eqref{eq:BE-square-detail}. The
centered part has conditional variance at most
\[
        C_V\eps^{-3}\sum_zB_E(z,k+2)^2
        \le C N_c\eps^{-5},
\]
which is $o(N_c^2\eps^{-4})$ since $N_c\eps\to\infty$. Hence
\[
        \sum_z W_zB_E(z,k+2)\le C N_c\eps^{-2}
\]
with high probability. Dividing by $S\ge c_SN_c\eps^{-1}$ proves
\eqref{eq:BE-input}, because $k\asymp\eps^{-1}$.

We prove \eqref{eq:near-input}. Let
\[
        A_{zw}:=\one_{\{0<\dist_\Core(z,w)\le k\}},
        \qquad
        T_{\rm near}:=\sum_{z,w}W_zW_wA_{zw}.
\]
Since every core vertex has degree at least two,
\begin{equation}\label{eq:Azw-total-detail}
        \sum_{z,w}A_{zw}
        \le \sum_z B_D(z,k).
\end{equation}
Also, from $\bar W_w\le C(\eps^{-1}+d_\Core(w))$ and $\eps<1$,
\[
        M_z:=\sum_w\bar W_wA_{zw}
        \le C\eps^{-1}B_D(z,k).
\]
Thus \eqref{eq:BD-square-detail} gives
\begin{equation}\label{eq:Mz-square-detail}
        \sum_zM_z^2\le C N_c\eps^{-4}.
\end{equation}
The deterministic near-pair contribution is bounded by
\[
        \sum_z\bar W_zM_z
        \le \Big(\sum_z\bar W_z^2\Big)^{1/2}
             \Big(\sum_zM_z^2\Big)^{1/2}
        \le C N_c\eps^{-3}.
\]
The linear fluctuation has conditional variance at most
\[
        4C_V\eps^{-3}\sum_zM_z^2\le C N_c\eps^{-7},
\]
and the quadratic fluctuation has conditional second moment at most
\[
        2C_V^2\eps^{-6}\sum_{z,w}A_{zw}
        \le C N_c\eps^{-7},
\]
where we used \eqref{eq:Azw-total-detail} and Cauchy's inequality with
\eqref{eq:BD-square-detail}. Both fluctuations are $o(N_c\eps^{-3})$ in
probability. Hence
\begin{equation}\label{eq:Tnear-detail}
        T_{\rm near}\le C N_c\eps^{-3}
\end{equation}
with high probability.

The diagonal contribution is separate. By \eqref{eq:W-constants-detail} and
conditional independence,
\[
        \Egr\left[\sum_zW_z^2\mid\Core\right]\le C N_c\eps^{-3},
        \qquad
        \Vargr\left(\sum_zW_z^2\mid\Core\right)\le C N_c\eps^{-7}.
\]
Since $N_c\eps\to\infty$, Chebyshev's inequality gives
\begin{equation}\label{eq:diagonal-weight-detail}
        \sum_zW_z^2\le C N_c\eps^{-3}
\end{equation}
with high probability. Combining \eqref{eq:Tnear-detail} and
\eqref{eq:diagonal-weight-detail} and dividing by
$S^2\ge c_S^2N_c^2\eps^{-2}$ gives
\[
        \sum_{z,w}\nu(z)\nu(w)\one_{\{\dist_\Core(z,w)\le k\}}
        \le \frac{C}{N_c\eps}.
\]
Since $N_c\ge c_{\rm core}\eps^2n$ on the DKLP structural event, the right side is at
most $C_{\rm near}/(\eps^3n)$. This proves \eqref{eq:near-input}.
\end{proof}

\section{Proof of Lemma~\ref{lem:dklp-meeting}}\label{sec:proof}

We prove the lemma on the intersection of the high-probability events in
Lemmas~\ref{lem:standard}, \ref{lem:avg-target}, and~\ref{lem:core-balls}. After
enlarging constants if needed, all constants below are deterministic on this
event.

\subsection*{Upper bound}

By Lemma~\ref{lem:avg-target}, on the present event
\[
        \Kemeny(H)=\sum_v\pi(v)\E^H_\pi\tau_v\le C_0n
\]
for an absolute constant $C_0$. Lemma~\ref{lem:aldous-meeting} and the mixing
estimate \eqref{eq:mixing} give
\[
        \max_{x,y\in V(H)}\E^H_{x,y}\meet
        \le K'_A\bigl(\tmix(H)+\Kemeny(H)\bigr)
        \le Cn,
\]
since $\tmix(H)=o(n)$. This proves the upper bound.

\subsection*{Lower bound}

We prove the lower bound from stationarity. The test function sees only the two
projected roots in the core $\Core$. Motion inside attached trees is invisible.

Let $\cE_H$ denote the Dirichlet form of the independent two-walk chain on the
fixed graph $H$. Since each walk waits an exponential time of mean one, the
transition rate from $(x,y)$ to $(x',y)$ is $1/d(x)$ when $x'\sim x$, and
similarly for moves of the second coordinate. Thus, for a function $F$ on
$V(H)^2$,
\begin{equation}\label{eq:dirichlet-form}
\begin{aligned}
        \cE_H(F,F)
        =\frac12\sum_{x,y}\pi(x)\pi(y)
        \Bigg[&\sum_{x'\sim x}\frac{\bigl(F(x,y)-F(x',y)\bigr)^2}{d(x)} \\
              &+\sum_{y'\sim y}\frac{\bigl(F(x,y)-F(x,y')\bigr)^2}{d(y)}\Bigg].
\end{aligned}
\end{equation}

Take $\mathbf{B}=\Diag:=\{(x,x):x\in V(H)\}$. Set $k=\lfloor\eps^{-1}\rfloor$. Define
the core-closeness score by
\begin{equation}\label{eq:test-F}
        F(x,y):=\max\left\{1-\frac{\dist_{\Core}(r(x),r(y))}{k},\,0\right\}.
\end{equation}
This function is $1$ when the walkers meet. It may also be $1$ for two distinct
vertices in the same attached tree, but that only enlarges the test set and does
not hurt the bound. Such pairs are covered by the near-projection estimate
below. The score is supported on pairs whose core projections are within
distance $k$, and it is $1/k$-Lipschitz in each core coordinate.

The test function has small stationary mean. Indeed,
\begin{equation}\label{eq:F-mean}
        \E^H_{\pi\otimes\pi} F
        \le \sum_{z,w\in V(\Core)}\nu(z)\nu(w)
        \one_{\{\dist_{\Core}(z,w)\le k\}}
        =o(1)
\end{equation}
by \eqref{eq:near-input}. Two independent stationary projections to the core
are unlikely to be within distance $O(\eps^{-1})$.

The test function also has small Dirichlet energy. The function $F$ is constant
inside each attached tree, and it is unchanged by a jump between a tree root and
a tree vertex. Only jumps along edges of the core $\Core$ can change $F$. If
one walker moves across a core edge $zz'$ and the other walker has core
projection $w$, then
\[
        |F(z,w)-F(z',w)|\le k^{-1}.
\]
This difference can be nonzero only when the edge $zz'$ has an endpoint within
core distance $k+2$ of $w$. Indeed, if the function changes, at least one of
$z$ and $z'$ lies within distance $k$ of $w$. Since $z$ and $z'$ are adjacent,
this edge is counted by $B_E(w,k+2)$.

Substituting the preceding Lipschitz and support estimates into \eqref{eq:dirichlet-form} gives the following bookkeeping. For a first-coordinate
jump across $x\sim x'$ while the second coordinate is $y$, the weight before the
outside factor $1/2$ is
\[
        \pi(x)\pi(y)\frac1{d(x)}=\frac{\pi(y)}{2|E(H)|}.
\]
If the second coordinate has core projection $w$, summing over all such $y$
gives
\[
        \sum_{r(y)=w}\frac{\pi(y)}{2|E(H)|}
        =\frac{\nu(w)}{2|E(H)|}.
\]
Pairing the two orientations of an undirected core edge cancels the factor
$1/2$ in the Dirichlet form. The second coordinate contributes the same amount.
Equivalently, loops aside, which contribute zero, the bookkeeping gives the
exact core expression
\[
        \cE_H(F,F)
        =\frac1{|E(H)|}\sum_{\{z,z'\}\in E(\Core)}
          \sum_{w\in V(\Core)}\nu(w)\bigl(F(z,w)-F(z',w)\bigr)^2.
\]
Here core edges are counted with multiplicity. Hence
\begin{equation}\label{eq:F-energy}
\begin{aligned}
        \cE_H(F,F)
        &\le \frac{1}{|E(H)|k^2}\sum_{w\in V(\Core)}\nu(w)B_E(w,k+2) \\
        &\le \frac{C}{k^2}\cdot\frac{k}{|E(H)|}
         =O\!\left(\frac1n\right),
\end{aligned}
\end{equation}
where we used \eqref{eq:BE-input}, $k=\Theta(\eps^{-1})$, and
$|E(H)|=\Theta(\eps n)$.

Lemma~\ref{lem:stationary-variational}, \eqref{eq:F-mean}, and
\eqref{eq:F-energy} now give
\[
        \E^H_{\pi\otimes\pi}\meet=\E^H_{\pi\otimes\pi}\tau_\Diag\ge c n.
\]
Also,
\[
        \max_{x,y}\E^H_{x,y}\meet\ge \E^H_{\pi\otimes\pi}\meet,
\]
so the lower bound in \eqref{eq:main-bounds} follows. This completes the proof
of Lemma~\ref{lem:dklp-meeting}.

\section{Proofs of the meeting-time bounds}\label{sec:remaining-proofs}

\begin{proof}[Proof of Theorem~\ref{thm:main-results}(i)]
Fix the constants $c,C$ from Lemma~\ref{lem:dklp-meeting}. Let
$\mathcal B_n\subset\mathfrak G$ consist of all empty or disconnected
multigraphs, together with all nonempty connected multigraphs $G$ for which
\[
        c n
        \le \E^{G}_{\pi\otimes\pi}\meet
        \le \max_{x,y\in V(G)}\E^{G}_{x,y}\meet
        \le Cn
\]
fails. This event is measurable and invariant under relabelling, because the
walk, stationary measure, and meeting-time expectations are deterministic
functions of the sampled multigraph under the conventions fixed above. Self-loops
cause no separate difficulty: a length-one self-loop is just a self-transition
for this deterministic walk, while the actual Erd\H{o}s--R\'enyi component lies
in the subspace without loops. Lemma~\ref{lem:dklp-meeting} gives
$\Pbb_{\rm DKLP}(\mathcal B_n)\to0$. We use DKLP contiguity in the required
direction: every isomorphism-invariant graph property holding with high probability in the
contiguous DKLP model also holds with high probability for the actual giant
component $\mathcal C_1$ of $G(n,(1+\eps)/n)$ \cite[Theorem~2]{dklp}. Hence
$\Pbb_{\rm ER}(\mathcal B_n)\to0$.
\end{proof}

\begin{lemma}[Parameter-retaining fixed-supercritical target bound]
\label{lem:fixed-parameter-target}
Fix $a>1$. Let $q=q(a)$ be the extinction probability of a $\Poi(a)$
Galton--Watson process, and put
\[
        \rho:=1-q,
        \qquad \mu:=aq,
        \qquad \delta:=1-\mu,
        \qquad \Lambda_0:=a-\mu=a\rho .
\]
Since $q=e^{-a(1-q)}$, we have
$\mu e^{-\mu}=ae^{-a}$ and $\mu\in(0,1)$. Thus this definition of $\mu$
agrees with the conjugate-parameter definition in
\eqref{eq:mu-conjugate}, with $a$ in place of $1+\eps$.

Let $H^{(a)}$ be the strictly supercritical contiguous model of
Ding--Lubetzky--Peres \cite[Theorem~1]{dlp-strict}. More precisely, first
sample
\[
        \widehat\Lambda_n\sim\mathcal N(\Lambda_0,1/n).
\]
On the event $\{\widehat\Lambda_n\le0\}$, define the output to be a fixed
two-vertex graph with one edge. Otherwise take independent
$D_u\sim\Poi(\widehat\Lambda_n)$ for $u\in[n]$, retain the vertices with
$D_u\ge3$, condition their total degree to be even, and pair their half-edges
uniformly to obtain the kernel. Replace its edges by independent geometric
paths with success parameter $\delta$, and attach independent $\PGW(\mu)$
trees to the core vertices.

There is an absolute constant $C_{\rm par}<\infty$ such that, for every fixed
$a>1$ and every $\eta>0$, there exists $n_0=n_0(a,\eta)$ for which, whenever
$n\ge n_0$,
\begin{equation}\label{eq:parameter-retaining-probability}
 \Pbb\left(
 \sum_{u,v\in V(H^{(a)})}\pi(u)\pi(v)\Reff^{H^{(a)}}(u,v)
 \le \frac{C_{\rm par}}{\Lambda_0},
 \quad
 |E(H^{(a)})|\le C_{\rm par}\Lambda_0 n
 \right)\ge1-\eta .
\end{equation}
Consequently, on the same event,
$\Kemeny(H^{(a)})\le C_{\rm par}^2n$. In particular, the constant is uniform
in $a$, although the threshold $n_0(a,\eta)$ need not be.
\end{lemma}

\begin{proof}
The conjugacy identities can be written in terms of $\Lambda_0$ as
\[
        q=e^{-\Lambda_0},
        \qquad
        \mu=\frac{\Lambda_0}{e^{\Lambda_0}-1},
        \qquad
        \rho=1-e^{-\Lambda_0}.
\]
Consequently, if $m_0:=1\vee\Lambda_0$, then
\begin{equation}\label{eq:fixed-parameter-relations-new}
        \delta m_0\asymp\Lambda_0,\qquad
        \rho\asymp\delta,\qquad
        \mu m_0^2\le C,\qquad
        \frac{\mu}{\delta}\le\frac{C}{\Lambda_0},\qquad
        \frac{\mu}{\delta^2}\le\frac{C}{\Lambda_0^2},
\end{equation}
with absolute comparison constants. These estimates follow by Taylor
expansion at $\Lambda_0=0$ and by the exponential decay of $\mu$ when
$\Lambda_0\ge1$.

For the fixed value of $a$, the event
\[
        \mathcal A_n:=
        \left\{\frac{\Lambda_0}{2}\le\widehat\Lambda_n
        \le\frac{3\Lambda_0}{2}\right\}
\]
has probability tending to one. In particular, the exceptional event on which
the Gaussian parameter is nonpositive is negligible. Work on $\mathcal A_n$
and put $m:=1\vee\widehat\Lambda_n$, so that $m\asymp m_0$ and
$\delta m\asymp\Lambda_0$ with absolute constants.

Let $N=|V(\Kernel)|$, write $d_x=d_{\Kernel}(x)$, and put
$M=|E(\Kernel)|$. Conditional laws of large numbers for the retained
conditioned-Poisson degrees, together with
Lemma~\ref{lem:uniform-conditioned-poisson-expansion}, give, with probability
tending to one,
\begin{align}
        M&\asymp Nm,                                                   \label{eq:fixed-kernel-size}\\
        \sum_xd_x^r&\le C_rNm^r\quad(r=2,3),                            \label{eq:fixed-positive-degree-moments}\\
        \sum_xd_x^{-r}&\le C_rNm^{-r}\quad(r=2,4),                      \label{eq:fixed-inverse-degree-moments}\\
        \lambda_2(I-P_{\Kernel})&\ge\gamma_* .                         \label{eq:fixed-normalized-gap}
\end{align}
All constants here are absolute. The rates of convergence may depend on the
fixed $a$. The parity conditioning is harmless for these laws of large
numbers, and $N\to\infty$ in probability for every fixed $a>1$.

We next prove the weighted-core estimate that was implicit in the
parameter-retaining argument.

\begin{claim}[Parameter-uniform weighted stretched-core resistance]
\label{claim:fixed-weighted-core}
Let
\[
        D_c:=\sum_{z\in V(\Core)}d_{\Core}(z)=2|E(\Core)|.
\]
With probability tending to one,
\begin{align}
 \sum_{z,w}d_{\Core}(z)d_{\Core}(w)\Reff^{\Core}(z,w)
 &\le \frac{C D_c^2}{\Lambda_0},                                      \label{eq:fixed-Rdd1}\\
 \sum_{z,w}d_{\Core}(z)d_{\Core}(w)\Reff^{\Core}(z,w)^2
 &\le \frac{C D_c^2}{\Lambda_0^2},                                    \label{eq:fixed-Rdd2}
\end{align}
where $C$ is absolute.
\end{claim}

\begin{proof}[Proof of the claim]
The normalized gap in \eqref{eq:fixed-normalized-gap} and
$d_x\ge3$ imply an absolute lower bound $\gamma_0>0$ for the first nonzero
eigenvalue of the unnormalized Laplacian $\Delta_{\Kernel}$. They also give
the pointwise resistance estimate
\[
        r_{xy}:=\Reff^{\Kernel}(x,y)
        \le\gamma_*^{-1}\left(\frac1{d_x}+\frac1{d_y}\right).
\]
It follows from \eqref{eq:fixed-inverse-degree-moments} that, for $r=2,4$,
\begin{equation}\label{eq:fixed-unweighted-kernel-resistance-moments}
        \sum_{x,y}r_{xy}^r
        \le C_rN\sum_xd_x^{-r}
        \le C_rN^2m^{-r}.
\end{equation}

Let $L_e$ be the independent $\operatorname{Geom}(\delta)$ stretch lengths,
and let $\Kernel^L$ be the weighted kernel in which edge $e$ has resistance
$L_e$. For ordered kernel vertices $x,y$, let $\theta^{x,y}$ be the unit
electrical current from $x$ to $y$ in the unweighted kernel. Thomson's
principle and Jensen's inequality give, for $r=2,4$,
\begin{equation}\label{eq:fixed-stretched-jensen}
 \Reff^{\Kernel^L}(x,y)^r
 \le r_{xy}^{\,r-1}\sum_eL_e^r(\theta_e^{x,y})^2.
\end{equation}
Define
\[
        Q_{e,r}:=\sum_{x,y}r_{xy}^{\,r-1}(\theta_e^{x,y})^2.
\]
Then
\[
        \sum_eQ_{e,r}=\sum_{x,y}r_{xy}^r\le C_rN^2m^{-r}.
\]
For an oriented non-loop edge $e$, let $b_e$ be its signed incidence vector.
The current representation and the unnormalized gap give
\[
 \sum_{x,y}(\theta_e^{x,y})^2
 =2N b_e^{\mathsf T}(\Delta_{\Kernel}^{\dagger})^2b_e
 \le\frac{2N}{\gamma_0}
       b_e^{\mathsf T}\Delta_{\Kernel}^{\dagger}b_e
 \le\frac{2N}{\gamma_0}.
\]
The expression is zero for a loop. Since $r_{xy}$ is bounded by an absolute
constant, this yields $\max_eQ_{e,r}\le C_rN$.

The geometric moments satisfy
$\E L_e^r\le C_r\delta^{-r}$ and
$\Var(L_e^r)\le C_r\delta^{-2r}$. Conditional on the kernel,
\begin{align*}
 \E\left[\sum_eL_e^rQ_{e,r}\,\middle|\,\Kernel\right]
 &\le \frac{C_rN^2}{\delta^rm^r},\\
 \Var\left(\sum_eL_e^rQ_{e,r}\,\middle|\,\Kernel\right)
 &\le \frac{C_rN^3}{\delta^{2r}m^r}.
\end{align*}
Because $m^r/N\to0$ for each fixed $a$, Chebyshev's inequality and
\eqref{eq:fixed-stretched-jensen} show that, simultaneously for $r=2,4$,
\begin{equation}\label{eq:fixed-stretched-resistance-moments}
        Z_r:=\sum_{x,y}\Reff^{\Kernel^L}(x,y)^r
        \le\frac{C_rN^2}{\delta^rm^r}
\end{equation}
with probability tending to one. This supplies both resistance moments and
their concentration explicitly.

Assign every core vertex $z$ to a nearest kernel endpoint $a(z)$, write
\[
        \ell(z):=\dist_{\Core}(z,a(z)),
        \qquad
        \omega_x:=\sum_{z:\,a(z)=x}d_{\Core}(z),
\]
and count loops by their two incident half-edges. Then
\[
        \sum_x\omega_x=D_c,
        \qquad
        \omega_x\le C\sum_{e\ni x}L_e.
\]
Put
\[
        Y_\omega:=\sum_{e=xy}(d_x+d_y)L_e^2.
\]
This is the edge sum that dominates the squared projected weights. Consequently,
\begin{equation}\label{eq:fixed-projected-weight-square}
 \sum_x\omega_x^2
 \le C Y_\omega
 \le \frac{CNm^2}{\delta^2}
\end{equation}
with probability tending to one. Indeed, conditional on the kernel,
\[
 \Egr[Y_\omega\mid\Kernel]
 \le C\delta^{-2}\sum_xd_x^2,
 \qquad
 \Vargr(Y_\omega\mid\Kernel)
 \le C\delta^{-4}\sum_xd_x^3.
\]
Equations
\eqref{eq:fixed-positive-degree-moments} and Chebyshev's inequality apply.
Also, by the geometric law of large numbers and
\eqref{eq:fixed-kernel-size},
\begin{equation}\label{eq:fixed-core-degree-size}
        D_c=2\sum_eL_e\asymp\frac{Nm}{\delta}.
\end{equation}

Put
\[
        A_j:=\sum_zd_{\Core}(z)\ell(z)^j,\qquad j=1,2.
\]
On a stretched edge of length $L$, the contributions are bounded by
$CL(L-1)$ for $j=1$ and by $CL^2(L-1)$ for $j=2$. The same bounds hold for
stretched self-loops. Moreover, the squared per-edge contributions have
expectations at most $C\mu\delta^{-4}$ and $C\mu\delta^{-6}$, respectively.
Since $M\mu\to\infty$ for every fixed $a$, the geometric factorial moments
and Chebyshev's inequality therefore give
\[
        A_1\le\frac{CM\mu}{\delta^2},
        \qquad
        A_2\le\frac{CM\mu}{\delta^3}
\]
with probability tending to one. Together with
\eqref{eq:fixed-parameter-relations-new} and
\eqref{eq:fixed-core-degree-size}, this gives
\begin{equation}\label{eq:fixed-core-endpoint-distance}
        A_1\le\frac{CD_c}{\Lambda_0},
        \qquad
        A_2\le\frac{CD_c}{\Lambda_0^2}.
\end{equation}

Cauchy's inequality now removes any need to compare the correlated projected
weights pointwise with kernel degrees:
\begin{align*}
 \sum_{x,y}\omega_x\omega_y\Reff^{\Kernel^L}(x,y)
 &\le\left(\sum_x\omega_x^2\right)Z_2^{1/2}
 \le\frac{CD_c^2}{\Lambda_0},\\
 \sum_{x,y}\omega_x\omega_y\Reff^{\Kernel^L}(x,y)^2
 &\le\left(\sum_x\omega_x^2\right)Z_4^{1/2}
 \le\frac{CD_c^2}{\Lambda_0^2}.
\end{align*}
Finally,
\[
 \Reff^{\Core}(z,w)
 \le \ell(z)+\Reff^{\Kernel^L}(a(z),a(w))+\ell(w).
\]
Summing this inequality with weights $d_{\Core}(z)d_{\Core}(w)$, and using
\eqref{eq:fixed-core-endpoint-distance}, proves
\eqref{eq:fixed-Rdd1}. Squaring, using
$(x+y+z)^2\le3(x^2+y^2+z^2)$, proves
\eqref{eq:fixed-Rdd2}.
\end{proof}

We now insert the degree masses of the attached trees. Conditional on the core,
the variables
\[
        W_z:=\sum_{u:\,r(u)=z}d_{H^{(a)}}(u)
\]
are independent. Write
\[
        \bar W_z:=\Egr[W_z\mid\Core],
        \qquad \xi_z:=W_z-\bar W_z,
        \qquad S:=\sum_zW_z=2|E(H^{(a)})|.
\]
The parameter-retaining Galton--Watson estimates give
\begin{equation}\label{eq:fixed-W-moments-detailed}
        \bar W_z\le C\bigl(d_{\Core}(z)+\delta^{-1}\bigr)
        \le C\delta^{-1}d_{\Core}(z),
        \qquad
        \Vargr(\xi_z\mid\Core)\le C\delta^{-3},
        \qquad
        \Egr[\xi_z^4\mid\Core]\le C\delta^{-7}.
\end{equation}

For completeness, let
\[
        \psi(x):=\E[\Poi(x)\one_{\{\Poi(x)\ge3\}}]
             =x\bigl[1-e^{-x}(1+x)\bigr].
\]
The exact identities above give
$\psi(\Lambda_0)=\Lambda_0\rho\delta$. On $\mathcal A_n$,
\[
        \psi(\widehat\Lambda_n)\asymp\psi(\Lambda_0)
\]
with absolute comparison constants. Indeed, $\psi(x)\asymp x^3$ for
$0<x\le1$ and $\psi(x)\asymp x$ for $x\ge1$. Thus the retained kernel degree
sum is of order $n\psi(\widehat\Lambda_n)\asymp n\Lambda_0\rho\delta$.
The degree law of large numbers on $\mathcal A_n$, followed by the geometric
and Galton--Watson laws of large numbers, therefore yields
\begin{equation}\label{eq:fixed-S-Dc-relations}
        D_c\asymp n\Lambda_0\rho,
        \qquad
        S\asymp\frac{D_c}{\delta}\asymp\Lambda_0n,
        \qquad
        S\ge c\delta^{-1}D_c,
        \qquad
        \delta D_c\longrightarrow\infty
\end{equation}
with probability tending to one. To see the middle comparison directly, write
$I=\sum_e(L_e-1)$ and $V_c=N+I$. Then
$D_c=2(M+I)$, $V_c\le D_c/2$, and
\[
        \Egr[S\mid\Core]=D_c+\frac{2V_c\mu}{\delta}.
\]
Since $V_c\le D_c/2$, this conditional mean is at most
$D_c+D_c\mu/\delta=D_c/\delta$.
If $\delta\ge1/2$, the term $D_c$ gives the lower bound
$S\ge cD_c/\delta$. If $\delta<1/2$, the geometric law of large numbers gives
$V_c\ge cD_c$, and the tree term gives the same conclusion. The conditional
variance bound following from Lemma~\ref{lem:pgw-moments}, together with
$\delta D_c\to\infty$, proves concentration. All comparison constants in
\eqref{eq:fixed-S-Dc-relations} are absolute; only the convergence threshold
may depend on $a$.

Set $R_{zw}:=\Reff^{\Core}(z,w)$ and expand
\[
        T:=\sum_{z,w}W_zW_wR_{zw}=T_0+T_1+T_2,
\]
where
\[
        T_0:=\sum_{z,w}\bar W_z\bar W_wR_{zw},\qquad
        T_1:=2\sum_z\xi_z\sum_w\bar W_wR_{zw},\qquad
        T_2:=\sum_{z,w}\xi_z\xi_wR_{zw}.
\]
By \eqref{eq:fixed-Rdd1} and \eqref{eq:fixed-S-Dc-relations},
\begin{equation}\label{eq:fixed-deterministic-term}
        T_0
        \le C\delta^{-2}\frac{D_c^2}{\Lambda_0}
        \le C\frac{S^2}{\Lambda_0}.
\end{equation}
For the linear term, put $G_z:=\sum_w\bar W_wR_{zw}$. Since
$d_{\Core}(z)\ge2$, Cauchy's inequality and \eqref{eq:fixed-Rdd2} give
\begin{equation}\label{eq:fixed-G-square}
        \sum_zG_z^2
        \le C\delta^{-2}D_c
             \sum_{z,w}d_{\Core}(z)d_{\Core}(w)R_{zw}^2
        \le C\delta^{-2}\frac{D_c^3}{\Lambda_0^2}.
\end{equation}
Therefore, conditionally on the core,
\begin{equation}\label{eq:fixed-linear-var}
        \Vargr(T_1\mid\Core)
        \le C\delta^{-5}\frac{D_c^3}{\Lambda_0^2}.
\end{equation}
After division by $(S^2/\Lambda_0)^2$, the right-hand side is at most
$C/(\delta D_c)=o(1)$.

For the quadratic term, $R_{zz}=0$, and independence gives
\begin{equation}\label{eq:fixed-quadratic-second}
        \Egr[T_2^2\mid\Core]
        \le C\delta^{-6}\sum_{z,w}R_{zw}^2
        \le C\delta^{-6}\frac{D_c^2}{\Lambda_0^2}.
\end{equation}
The ratio of this bound to $(S^2/\Lambda_0)^2$ is at most
$C/(\delta^2D_c^2)=o(1)$. Thus
\begin{equation}\label{eq:fixed-tree-weighted-core-bound}
        \sum_{z,w}W_zW_w\Reff^{\Core}(z,w)
        \le C\frac{S^2}{\Lambda_0}
\end{equation}
with probability tending to one and with an absolute constant.

Finally, for arbitrary vertices $u,v\in H^{(a)}$,
\[
        \Reff^{H^{(a)}}(u,v)
        \le h(u)+\Reff^{\Core}(r(u),r(v))+h(v).
\]
For one attached $\PGW(\mu)$ tree, let $B_z$ be its degree-weighted depth.
Lemma~\ref{lem:pgw-moments} gives
$\Egr B_z\le C\mu\delta^{-2}$ and $\Egr B_z^2\le C\delta^{-5}$.
Conditional on the core, the $B_z$ are independent. Since
$V_c=\Theta_a(n)$ and hence $V_c\mu^2\delta\to\infty$ for each fixed $a$,
Chebyshev's inequality and \eqref{eq:fixed-S-Dc-relations} give
\begin{equation}\label{eq:fixed-depth-average}
        \sum_u\pi(u)h(u)
        \le C\frac{V_c\mu\delta^{-2}}{S}
        \le C\frac{\mu}{\delta}
        \le \frac{C}{\Lambda_0}
\end{equation}
with probability tending to one. Dividing
\eqref{eq:fixed-tree-weighted-core-bound} by $S^2$ and adding the two depth
contributions proves the resistance bound in
\eqref{eq:parameter-retaining-probability}. Finally,
$S=2|E(H^{(a)})|\le C\Lambda_0n$ by
\eqref{eq:fixed-S-Dc-relations}. The resistance identity
\eqref{eq:kemeny-resistance} proves the asserted bound on $\Kemeny$.
\end{proof}
\begin{proof}[Proof of Lemma~\ref{lem:kemeny}]
In the slightly supercritical regime, Lemma~\ref{lem:avg-target} gives
$\Kemeny(H)=O(n)$ in the DKLP model. The same contiguity transfer used in the
proof of Theorem~\ref{thm:main-results}(i) gives the estimate for the actual
giant component.

For the fixed-supercritical regime, apply
Lemma~\ref{lem:fixed-parameter-target} with $a=\lambda$ in the strictly
supercritical contiguous model. It gives $\Kemeny(H^{(\lambda)})\le Cn$ with
an absolute constant, while the threshold in $n$ is allowed to depend on the
fixed value of $\lambda$. The contiguity theorem
\cite[Theorem~1]{dlp-strict} transfers this graph property to the giant
component of $G(n,\lambda/n)$.
\end{proof}

\begin{lemma}[Fixed-supercritical mixing]\label{lem:fixed-mixing}
For every fixed $\lambda>1$, if $H_\lambda$ is the giant component of
$G(n,\lambda/n)$, then
\[
        \tmix(H_\lambda)=O_\lambda(\log^2 n)=o_\lambda(n)
\]
with high probability.
\end{lemma}

\begin{proof}
This is Theorem~1.1 of Benjamini--Kozma--Wormald \cite{bkw}, translated from
lazy discrete time to the jump-rate-$1$ continuous-time convention as in the
proof of Lemma~\ref{lem:standard}.
\end{proof}

\begin{lemma}[Uniform lower bound near criticality]
\label{lem:uniform-near-lower}
There are absolute constants $\lambda_0>1$ and $c_0>0$ such that, for every
fixed $\lambda\in(1,\lambda_0]$, if $H_\lambda$ is the giant component of
$G(n,\lambda/n)$, then
\[
        \Pbb\left(
        \E^{H_\lambda}_{\pi\otimes\pi}\meet\ge c_0n
        \right)\longrightarrow1.
\]
\end{lemma}

\begin{proof}
Let $\mu_\lambda<1$ be conjugate to $\lambda$ and put
$\delta_\lambda=1-\mu_\lambda$. Choose $\lambda_0>1$ so close to one that
$\delta_\lambda\asymp\lambda-1$ with absolute constants for
$1<\lambda\le\lambda_0$. We first work in the strictly supercritical contiguous
model \cite[Theorem~1]{dlp-strict} and set
$k_\lambda=\lfloor\delta_\lambda^{-1}\rfloor$.

The proofs of Lemmas~\ref{lem:kernel-ball-exploration} and \ref{lem:core-balls}
use only the moment bounds for geometric stretches and $\PGW(\mu)$ trees, the
expansion of the conditioned-Poisson kernel, and the relations between the
kernel, core, and giant sizes. Over $1<\lambda\le\lambda_0$ these inputs have
the same absolute constants after replacing $\eps$ by $\delta_\lambda$. Hence,
with high probability,
\begin{equation}\label{eq:fixed-near-core-estimates}
 \sum_{z,w}\nu(z)\nu(w)
 \one_{\{\dist_{\Core}(z,w)\le k_\lambda\}}=o_\lambda(1),
 \qquad
 \sum_z\nu(z)B_E(z,k_\lambda+2)\le Ck_\lambda,
\end{equation}
where $C$ is absolute. Also
$|E(H_\lambda)|\asymp(\lambda-1)n$ and
$k_\lambda\asymp(\lambda-1)^{-1}$ with absolute constants.

Use the same core-projection test as in \eqref{eq:test-F}, with
$k=k_\lambda$. The first estimate in \eqref{eq:fixed-near-core-estimates} makes
its stationary mean $o_\lambda(1)$. The second estimate in \eqref{eq:fixed-near-core-estimates} and the
Dirichlet-form calculation \eqref{eq:F-energy} give
\[
        \cE_{H_\lambda}(F,F)
        \le \frac{C}{|E(H_\lambda)|k_\lambda}
        \le \frac{C'}n
\]
with absolute constants. Lemma~\ref{lem:stationary-variational} gives
$\E_{\pi\otimes\pi}^{H_\lambda}\meet\ge c_0n$. Contiguity transfers the
estimate to the actual giant component.
\end{proof}

\begin{lemma}[Uniform diagonal lower bound away from criticality]
\label{lem:uniform-away-lower}
For the fixed $\lambda_0>1$ in Lemma~\ref{lem:uniform-near-lower}, there is an
absolute constant $c_1>0$ such that, for every fixed $\lambda\ge\lambda_0$,
\[
        \Pbb\left(
        \E^{H_\lambda}_{\pi\otimes\pi}\meet\ge c_1n
        \right)\longrightarrow1.
\]
\end{lemma}

\begin{proof}
Let $G=H_\lambda$ and take the indicator of the diagonal,
\[
        F(x,y)=\one_{\{x=y\}}.
\]
There are absolute constants $A,b>0$, depending only on the already fixed number
$\lambda_0$, such that uniformly over fixed $\lambda\ge\lambda_0$, with high
probability,
\[
        \sum_{x\in V(G)}d(x)^2\le A(\lambda^2+\lambda)n,
        \qquad
        |E(G)|\ge b\lambda n.
\]
The degree estimate follows from exponential-tail bounds for binomial degrees,
and the edge estimate follows from the explicit giant-edge density; see, for
example, \cite[Chapter~4]{vanderhofstad}. Hence
\[
        \E_{\pi\otimes\pi}F
        =\sum_x\pi(x)^2
        =\frac{\sum_xd(x)^2}{(2|E(G)|)^2}
        \le \frac{C}{n}.
\]
A jump of either coordinate can change $F$ only by entering or leaving the
diagonal. Reversibility gives
\[
        \cE_G(F,F)\le2\sum_x\pi(x)^2\le\frac{C}{n}.
\]
Lemma~\ref{lem:stationary-variational} yields the claim, with an absolute $c_1$.
\end{proof}

\begin{lemma}[Separated mass from effective resistance]\label{lem:separated-mass}
Let $G$ be a finite connected multigraph. Suppose that there are nonempty
disjoint sets $A,B\subseteq V(G)$ and constants $\alpha,R>0$ such that
\[
        \pi(A)\ge\alpha,\qquad \pi(B)\ge\alpha,\qquad
        \Reff(A,B)\ge R .
\]
Then
\[
        \E_{\pi\otimes\pi}\meet\ge 4\alpha^4|E(G)|R .
\]
\end{lemma}

\begin{proof}
Let $f$ be the equilibrium potential between $A$ and $B$ \cite[Chapter~9]{lpw}:
$f=0$ on $A$, $f=1$ on $B$, and
\[
        \frac1{2|E(G)|}\sum_{uv\in E(G)}(f(u)-f(v))^2
        =\frac1{2|E(G)|\,\Reff(A,B)}.
\]
By the maximum principle, $0\le f\le1$. Set $F(x,y)=1-|f(x)-f(y)|$. Then
$0\le F\le1$ and $F=1$ on the diagonal. Also $F=0$ on $A\times B$ and on
$B\times A$, so
\[
        1-\E_{\pi\otimes\pi}F\ge 2\pi(A)\pi(B)\ge2\alpha^2.
\]
The map $(a,b)\mapsto1-|a-b|$ is one-Lipschitz in each coordinate. Hence the
Dirichlet form of $F$ for the two-walk chain is at most twice the one-walk
Dirichlet form of $f$:
\[
        \cE_G(F,F)
        \le 2\cdot \frac1{2|E(G)|}
             \sum_{uv\in E(G)}(f(u)-f(v))^2
        \le \frac1{|E(G)|R}.
\]
Lemma~\ref{lem:stationary-variational} gives the result.
\end{proof}

\begin{lemma}[Critical-window geometric input]
\label{lem:critical-geometry}
Fix $A<\infty$ and suppose that the edge probabilities $p_n$ satisfy
$|n^{1/3}(np_n-1)|\le A$.  Let $\mathcal C_1$ be the largest component of $G(n,p_n)$.
For every $\eta>0$ there are constants
$0<b_{A,\eta}<B_{A,\eta}<\infty$ and $\alpha_{A,\eta}>0$ such that, with
probability at least $1-\eta$ for all large $n$,
\begin{align*}
        b_{A,\eta} n^{2/3}&\le |E(\mathcal C_1)|\le B_{A,\eta} n^{2/3},\\
        \tmix(\mathcal C_1)&\le B_{A,\eta} n,\\
        \max_{u,v\in V(\mathcal C_1)}\Reff^{\mathcal C_1}(u,v)&\le B_{A,\eta} n^{1/3},
\end{align*}
and there are disjoint vertex sets $A_n,B_n\subseteq V(\mathcal C_1)$ such that
\[
        \pi(A_n),\pi(B_n)\ge \alpha_{A,\eta},
        \qquad
        \Reff^{\mathcal C_1}(A_n,B_n)\ge b_{A,\eta} n^{1/3}.
\]
\end{lemma}

\begin{proof}
First choose $\beta=\beta(A,\eta)>0$ so that, for all sufficiently large $n$,
\[
        \Pbb\bigl(|\mathcal C_1|<\beta n^{2/3}\bigr)\le\eta/4;
\]
this is the critical-window component-size estimate quoted after
Theorem~1.2 of \cite{nachmias-peres-critical}. Next choose the constants in
the diameter, mixing-time, and edge-volume upper estimates of Nachmias and
Peres so that their joint failure probability is at most $\eta/4$. Thus their
intersection with the preceding component-size event has failure probability
at most $\eta/2$. Their results apply throughout
$p=(1+O(n^{-1/3}))/n$, and on this event
\[
 \operatorname{diam}(\mathcal C_1)=O_{A,\eta}(n^{1/3}),\qquad
 \tmix(\mathcal C_1)=O_{A,\eta}(n),\qquad
 |E(\mathcal C_1)|=O_{A,\eta}(n^{2/3})
\]
\cite[Theorems~1.1--1.2 and Section~6]{nachmias-peres-critical}. Since
$\mathcal C_1$ is connected, this event also gives, for all sufficiently large
$n$,
\[
        |E(\mathcal C_1)|
        \ge |\mathcal C_1|-1
        \ge \frac{\beta}{2}n^{2/3}.
\]
The
conversion from lazy discrete time to the jump-rate-$1$ continuous-time
normalization changes the mixing estimate only by an absolute constant, as in
Lemma~\ref{lem:standard}. Rayleigh monotonicity and
$\Reff(u,v)\le\dist(u,v)$ give the resistance upper bound.

For the separated-mass estimate, apply the lower-bound construction in
\cite[Lemma~5.4, Propositions~5.5--5.7 and proof of
Theorem~2.1(c.2)]{nachmias-peres-critical}, choosing its constants so that its
failure probability is at most $\eta/2$. On the resulting event, every
component $C$ with at least $\beta n^{2/3}$ vertices contains a vertex $v$,
radii $h<k<r$, and an integer $L=O_{A,\eta}(1)$ such that
\[
        |B(v,h)|\ge c n^{2/3},\qquad
        |E(B(v,r))|\le |E(C)|/3,\qquad
        \frac{k}{L}\ge c n^{1/3},\qquad
        h<\frac{k}{4L},
\]
and $v$ is not $L$-lane-rich for $(k,r)$. Here $E(B(v,r))$ denotes the set
of edges of $C$ with both endpoints in $B(v,r)$.

By definition, the last condition gives at least $k/4-O(1)$ levels
\[
        j\in[\lceil k/2\rceil,k]
\]
having fewer than $L$ lanes for $(v,r)$. For each such $j$, let $\Pi_j$ be
the set of lane edges between levels $j-1$ and $j$. Since
$h<k/(4L)\le k/2$, every $\Pi_j$ separates $B(v,h)$ from
$C\setminus B(v,r-1)$. The cutsets are edge-disjoint, and
$1\le|\Pi_j|<L$. The Nash--Williams inequality \cite{nash-williams} therefore
gives
\[
        \Reff^C\bigl(B(v,h),\,C\setminus B(v,r-1)\bigr)
        \ge \sum_j\frac1{|\Pi_j|}
        \ge c\,\frac{k}{L}
        \ge c n^{1/3}.
\]

Apply this construction to $C=\mathcal C_1$ on the event
$|\mathcal C_1|\ge\beta n^{2/3}$, and set
\[
        A_n:=B(v,h),\qquad
        B_n:=\mathcal C_1\setminus B(v,r-1).
\]
The total edge volume of $\mathcal C_1$ is at most
$B_{A,\eta}n^{2/3}$ on the upper-estimate event, while $B(v,h)$ is connected
and has at least $cn^{2/3}$ vertices. Hence its degree volume is at least
$2(cn^{2/3}-1)$, and $\pi(A_n)\ge\alpha_{A,\eta}$ after decreasing
constants. Every edge incident to $B(v,r-1)$ has its other endpoint in
$B(v,r)$, so
\[
        \operatorname{vol}(B(v,r-1))
        \le2|E(B(v,r))|
        \le\frac23|E(\mathcal C_1)|.
\]
Dividing by $2|E(\mathcal C_1)|$ gives
$\pi(B(v,r-1))\le1/3$, and hence $\pi(B_n)\ge2/3$. The set-to-set resistance
bound above gives the required separation.

The joint component-size and upper-estimate event, and the separated-mass
event, each fail with probability at most $\eta/2$. Their intersection has
probability at least $1-\eta$, proving the lemma.
\end{proof}

\begin{proof}[Proof of Theorem~\ref{thm:main-results}(ii)]
Fix $\lambda>1$, and write $H_\lambda$ for the largest component of
$G(n,\lambda/n)$. Lemma~\ref{lem:kemeny} gives the uniform estimate
$\Kemeny(H_\lambda)\le C_{\rm K}n$ with high probability, while
Lemma~\ref{lem:fixed-mixing} gives $\tmix(H_\lambda)=o_\lambda(n)$. Hence, for
all sufficiently large $n$---with the threshold allowed to depend on the fixed
$\lambda$---we have $\tmix(H_\lambda)\le n$. Lemma~\ref{lem:aldous-meeting}
yields
\[
        t_{\rm meet}(H_\lambda)
        \le K'_A\bigl(\tmix(H_\lambda)+\Kemeny(H_\lambda)\bigr)
        \le K'_A(1+C_{\rm K})n.
\]
This gives an upper constant independent of $\lambda$.

For the lower bound, use Lemma~\ref{lem:uniform-near-lower} when
$1<\lambda\le\lambda_0$ and Lemma~\ref{lem:uniform-away-lower} when
$\lambda\ge\lambda_0$. With $c:=\min\{c_0,c_1\}$,
\[
        \E^{H_\lambda}_{\pi\otimes\pi}\meet\ge cn
\]
with high probability for every fixed $\lambda>1$. The worst-case expectation is
at least this stationary average. The same absolute constants therefore work in
both supercritical regimes after decreasing $c$ and increasing $C$ if needed.
\end{proof}

\begin{proof}[Proof of Theorem~\ref{thm:main-results}(iii)]
Fix $\eta>0$ and apply Lemma~\ref{lem:critical-geometry} with parameter
$\eta$. On the event in that lemma,
\[
        |E(\mathcal C_1)|\le B_{A_0,\eta} n^{2/3},\qquad
        \max_{u,v}\Reff(u,v)\le B_{A_0,\eta} n^{1/3}.
\]
By \eqref{eq:kemeny-resistance},
\[
        \Kemeny(\mathcal C_1)
        =|E(\mathcal C_1)|\sum_{u,v}\pi(u)\pi(v)\Reff(u,v)
        \le B_{A_0,\eta}^2 n .
\]
Together with $\tmix(\mathcal C_1)\le B_{A_0,\eta} n$, Aldous's comparison gives
\[
        t_{\rm meet}(\mathcal C_1)
        \le K'_A\bigl(\tmix(\mathcal C_1)+\Kemeny(\mathcal C_1)\bigr)
        \le C_{A_0,\eta} n .
\]
For the lower bound, the separated sets supplied by
Lemma~\ref{lem:critical-geometry} give
\[
        |E(\mathcal C_1)|\ge b_{A_0,\eta} n^{2/3},\qquad
        \Reff(A_n,B_n)\ge b_{A_0,\eta} n^{1/3},\qquad
        \pi(A_n),\pi(B_n)\ge\alpha_{A_0,\eta}.
\]
Lemma~\ref{lem:separated-mass} gives
\[
        \E^{\mathcal C_1}_{\pi\otimes\pi}\meet
        \ge 4\alpha_{A_0,\eta}^4 b_{A_0,\eta}^2 n.
\]
Changing constants and taking the intersection of the high-probability events
proves the theorem.
\end{proof}

\section{Coalescence time bounds: Proof of Corollary~\ref{cor:voter}}\label{sec:proofCoro}

We first prove the auxiliary lemmas needed for Corollary~\ref{cor:voter}.
\begin{lemma}[Slightly supercritical maximal hitting bound]\label{lem:slightly-supercritical-hit}
Let $H_{\rm ER}$ be the giant component of $G(n,(1+\eps)/n)$, where
$\eps\to0$ and $\eps^3n\to\infty$. Then, with high probability,
\begin{equation}\label{eq:slightly-supercritical-hit}
        t_{\rm hit}^{\rm ct}(H_{\rm ER})
        :=\max_{x,y\in V(H_{\rm ER})}\E_x^{H_{\rm ER}}\tau_y
        \le C_{\rm hit}n\log(\eps^3n)
\end{equation}
for an absolute constant $C_{\rm hit}<\infty$.
\end{lemma}

\begin{proof}
The diameter estimate of Ding--Kim--Lubetzky--Peres gives, throughout the
slightly supercritical regime,
\[
        \operatorname{diam}(H_{\rm ER})
        \le C_{\rm diam}\eps^{-1}\log(\eps^3n)
\]
with high probability \cite{dklp-diameter}. Also
$|E(H_{\rm ER})|=O(\eps n)$ with high probability by the DKLP size estimate
and contiguity. On the intersection of these events, Rayleigh monotonicity gives
$\Reff(x,y)\le \dist(x,y)\le\operatorname{diam}(H_{\rm ER})$ for all vertices.
The continuous-time commute identity yields
\[
        \E_x\tau_y+\E_y\tau_x
        =2|E(H_{\rm ER})|\Reff(x,y)
        \le C n\log(\eps^3n),
\]
and nonnegativity of the two hitting expectations gives \eqref{eq:slightly-supercritical-hit}.
\end{proof}

\begin{lemma}[Oliveira coalescence--hitting comparisons]\label{lem:oliveira-hit}
Let $(X_t)_{t\ge0}$ be a finite irreducible reversible continuous-time chain with
generator $\mathcal L=P-I$, where $P$ is a reversible transition kernel. Let
\[
        t_{\rm hit}^{\rm ct}:=\max_{x,y}\E_x\tau_y .
\]
For the standard continuous-time coalescing process driven by this chain, write
$C_k^{\rm coal}$ for the first time at which at most $k$ particles remain,
started from an arbitrary finite initial configuration. There is a universal
constant $C_{\rm Oli}<\infty$ such that, for all $k\ge1$,
\begin{equation}\label{eq:oliveira-partial}
        \E C_k^{\rm coal}
        \le C_{\rm Oli}\left(\frac{t_{\rm hit}^{\rm ct}}{k}+\tmix\right).
\end{equation}
In particular, for the process started with one particle at every state,
\begin{equation}\label{eq:oliveira-hit}
        \E\coal\le C_{\rm Oli}\,t_{\rm hit}^{\rm ct}.
\end{equation}
\end{lemma}

\begin{proof}
The full-coalescence estimate \eqref{eq:oliveira-hit} is the continuous-time
reversible-chain form of Oliveira's universal coalescence bound
\cite[Theorem~1.1]{oliveira-coalescence}. The partial-coalescence estimate
\eqref{eq:oliveira-partial} is the corresponding continuous-time form of
\cite[Theorem~1.2]{oliveira-coalescence}. Oliveira's statements apply to
arbitrary finite reversible Markov-chain generators; applying them to the
rate-one generator $P-I$ and changing the fixed mixing threshold only changes
the universal constant.
\end{proof}

\begin{lemma}[KMS full-coalescence comparison for reversible lazy kernels]
\label{lem:kms-reversible-kernel}
Let $Q$ be a finite irreducible reversible transition kernel on $N\ge2$ states with
stationary law $\pi$, and assume $Q(x,x)\ge1/2$ for every state. Let
$t_{\rm meet}^Q$ be the worst-case expected meeting time of two independent
synchronous $Q$-walks and let $t_{\rm mix}^Q$ be the total-variation mixing time
with threshold $1/4$. For the synchronous coalescing $Q$-walk process started
with one particle at every state,
\[
        \E T_{\rm coal}^Q
        \le C\,t_{\rm meet}^Q
        \left(1+\sqrt{\frac{t_{\rm mix}^Q}{t_{\rm meet}^Q}}\log N\right)
\]
with a universal constant $C$.
\end{lemma}

\begin{proof}
For the lazy simple random walk on an unweighted graph this is
Kanade--Mallmann-Trenn--Sauerwald \cite[Theorem~1.1]{kms}. Their mixing time
uses
\[
        \bar d_Q(t):=\max_{x,y}
        \|Q^t(x,\cdot)-Q^t(y,\cdot)\|_{\TV}
\]
with threshold $1/e$, whereas we use
$d_Q(t):=\max_x\|Q^t(x,\cdot)-\pi\|_{\TV}$ with threshold $1/4$. The
inequalities $d_Q(t)\le\bar d_Q(t)\le2d_Q(t)$ and the standard
submultiplicative fixed-threshold comparison show that these conventions differ
by at most universal multiplicative constants.

Put
\[
        R:=2Q-I.
\]
Since $Q(x,x)\ge1/2$, the matrix $R$ is an irreducible transition kernel,
reversible with respect to $\pi$, and $Q=(I+R)/2$. Realize $R$ as the walk on
a weighted looped multigraph with conductances
\[
        c_{xy}:=\pi(x)R(x,y)\quad(x\ne y),\qquad
        c_{xx}:=\frac12\pi(x)R(x,x).
\]
A loop counts twice in the weighted degree. Reversibility makes the
conductances symmetric, and the weighted degree at $x$ is
\[
        \sum_{y\ne x}c_{xy}+2c_{xx}
        =\pi(x)\sum_yR(x,y)=\pi(x).
\]
Thus the base walk of this network has kernel $R$, and its usual
one-half-lazification has kernel exactly $Q$.

If all conductances are rational, multiplication by a common denominator
produces an unweighted looped multigraph with parallel edges. Edge choices are
read with multiplicity, and loops implement the diagonal transitions of $R$.
The trajectory-exposure, meeting, and immortal-process arguments in
\cite{kms} depend only on the resulting half-lazy transition kernel and
therefore apply to this multigraph walk.

For general conductances, choose symmetric rational conductances
$c_{xy}^{(m)}$ with the same support and
$c_{xy}^{(m)}\to c_{xy}$. Let $R_m$ be the corresponding base-walk kernel
and put $Q_m=(I+R_m)/2$. Then $Q_m$ is irreducible and reversible,
$Q_m(x,x)\ge1/2$, and $Q_m\to Q$ entrywise; its stationary law $\pi_m$
converges to $\pi$.

For $\xi\in(0,1)$ write $t_{\rm mix}^Q(\xi)$ for the mixing time with
$d_Q$-threshold $\xi$, and put $s=t_{\rm mix}^Q(1/8)$. Entrywise convergence
of $Q_m^s$ and convergence of $\pi_m$ imply that, for all sufficiently large
$m$,
\[
        t_{\rm mix}^{Q_m}(1/4)\le s
        \le C\,t_{\rm mix}^{Q}(1/4).
\]
The last inequality is the fixed-threshold comparison.

The two-walk process killed on the diagonal and the coalescing process killed
upon reaching one particle are finite absorbing Markov chains. Their transient
matrices converge entrywise. Since their spectral radii are strictly smaller
than one, the corresponding fundamental matrices $(I-K_m)^{-1}$ converge.
Thus the absorption-time vector converges for every initial pair and every
initial coalescing configuration. In particular,
\[
        t_{\rm meet}^{Q_m}\longrightarrow t_{\rm meet}^{Q},
        \qquad
        \E T_{\rm coal}^{Q_m}\longrightarrow \E T_{\rm coal}^{Q}.
\]
Apply the rational-conductance result to $Q_m$ and pass to the limit, using the
preceding mixing-time bound. Enlarging the universal constant proves the
lemma.
\end{proof}

\begin{lemma}[Continuous-time coalescence--meeting comparison]\label{lem:ct-kms}
Let $(X_t)_{t\ge0}$ be a finite irreducible reversible continuous-time chain on
$N\ge2$ states with generator $\mathcal L=P-I$, where $P$ is an irreducible
reversible transition kernel with stationary law $\pi$. Let
$t_{\rm meet}^{\rm ct}$ be
the worst-case expected meeting time of two independent copies, and let
$t_{\rm mix}^{\rm ct}$ be its total-variation mixing time. For the standard
continuous-time coalescing process driven by this chain and started with one
particle at every state,
\begin{equation}\label{eq:kms-continuous}
        \E\coal
        \le C_{\rm KMS}\,t_{\rm meet}^{\rm ct}
        \left(1+\sqrt{\frac{t_{\rm mix}^{\rm ct}}
        {t_{\rm meet}^{\rm ct}}}\log N\right),
\end{equation}
with a universal constant $C_{\rm KMS}$.
\end{lemma}

\begin{proof}
Fix an observation mesh $\delta\in(0,\log2]$, say $\delta=1/2$, and put
$Q=e^{\delta\mathcal L}$. The kernel $Q$ is reversible with stationary law
$\pi$, and $Q(x,x)\ge e^{-\delta}\ge1/2$ because the rate-one update clock
has no ring during an interval of length $\delta$.

By Lemma~\ref{lem:kms-reversible-kernel}, the synchronous coalescent for $Q$
satisfies
\[
        \E T_{\rm coal}^Q
        \le C\,t_{\rm meet}^Q
        \left(1+\sqrt{\frac{t_{\rm mix}^Q}{t_{\rm meet}^Q}}\log N\right).
\]
Since $Q^k=e^{k\delta\mathcal L}$ and both mixing times use threshold $1/4$,
monotonicity of total variation gives the exact mesh comparison
\begin{equation}\label{eq:mesh-mixing-full}
        t_{\rm mix}^{\rm ct}
        \le\delta t_{\rm mix}^Q
        \le t_{\rm mix}^{\rm ct}+\delta .
\end{equation}

Mesh meetings are a subset of continuous-time meetings, and hence
\begin{equation}\label{eq:mesh-meeting-lower-full}
        t_{\rm meet}^{\rm ct}\le\delta t_{\rm meet}^Q.
\end{equation}
For the reverse bound, let $M$ be the worst expected continuous time until two
sampled paths agree at a mesh time. By the sampling construction,
$M=\delta t_{\rm meet}^Q$. From any starting pair, wait until the two
continuous-time paths first meet; this takes expected time at most
$t_{\rm meet}^{\rm ct}$. Conditional on that meeting, with probability at least
$p=e^{-2\delta}$ neither update clock rings before the next mesh time, so the
sampled paths agree there. If this attempt fails, restart at the next mesh time.
The strong Markov property gives
\[
        M\le t_{\rm meet}^{\rm ct}+\delta+(1-p)M,
\]
and therefore
\begin{equation}\label{eq:mesh-meeting-upper-full}
        \delta t_{\rm meet}^Q
        \le p^{-1}(t_{\rm meet}^{\rm ct}+\delta).
\end{equation}

We next couple the two coalescing systems. Let $S_k^Q$ be the occupied set of
the synchronous $Q$-coalescent after $k$ steps and let $S_t^{\rm ct}$ be the
occupied set of the continuous-time coalescent. Construct them inductively so
that
\[
        S_{k\delta}^{\rm ct}\subseteq S_k^Q\qquad(k\ge0).
\]
The inclusion is equality at time zero. Given it at time $k\delta$, generate,
independently for every vertex of $S_k^Q$, a continuous-time path segment of
length $\delta$ with generator $\mathcal L$. In the mesh process retain only
the endpoints and merge equal endpoints. These endpoints are independent
$Q$-steps, so this is exactly the synchronous $Q$-coalescent.

Particles of the continuous-time process use the same segments at the vertices
of $S_{k\delta}^{\rm ct}$. When two such segments meet, merge the particles
immediately and, according to a fixed priority ordering, retain the trajectory
of the higher-priority survivor. This is the standard graphical construction
of continuous-time coalescing walks. Every particle surviving to
$(k+1)\delta$ follows one of the segments issued from
$S_{k\delta}^{\rm ct}$, and distinct survivors have distinct endpoints. Hence
their occupied set is contained in the set of distinct endpoints of all
segments issued from $S_k^Q$, proving the induction. In particular,
\begin{equation}\label{eq:mesh-coalescence-domination}
        T_{\rm coal}^{\rm ct}\le\delta T_{\rm coal}^Q
        \quad\text{pathwise}.
\end{equation}

Combining the KMS estimate with
\eqref{eq:mesh-mixing-full}, \eqref{eq:mesh-meeting-upper-full}, and
\eqref{eq:mesh-coalescence-domination} gives
\[
 \E T_{\rm coal}^{\rm ct}
 \le C\left(
 t_{\rm meet}^{\rm ct}+\delta+
 \sqrt{(t_{\rm meet}^{\rm ct}+\delta)
       (t_{\rm mix}^{\rm ct}+\delta)}\,\log N\right).
\]
For $N\ge2$, both the worst-case meeting time and the mixing time are bounded
below by positive universal constants under the rate-one normalization.
Absorbing the fixed additive terms proves \eqref{eq:kms-continuous}.
\end{proof}

\begin{lemma}[KMS small-particle lemma in discrete time]
\label{lem:kms-small-discrete}
Let $Q$ be a finite irreducible reversible transition kernel on $N\ge2$ states
with stationary law $\pi$, and assume $Q(x,x)\ge1/2$ for all states. Let
$t_{\rm meet}^Q$ be the worst-case expected meeting time of two independent
synchronous $Q$-walks, and let $t_{\rm mix}^Q$ be the total-variation mixing
time with threshold $1/4$. Put
\[
        \alpha_Q:=\frac{t_{\rm meet}^Q}{t_{\rm mix}^Q}.
\]
There are universal constants $a_0,c_0,C_0\in(0,\infty)$ such that, if
$\alpha_Q\ge a_0$, then the expected coalescence time of the synchronous
coalescing $Q$-walk process started from any initial configuration with at most
$c_0\alpha_Q\log\alpha_Q$ particles is at most $C_0t_{\rm meet}^Q$.
\end{lemma}

\begin{proof}
For the lazy walk on an unweighted graph this is
Kanade--Mallmann-Trenn--Sauerwald \cite[Lemma~3.7]{kms}. As in the proof of
Lemma~\ref{lem:kms-reversible-kernel}, their pairwise mixing convention is
universally equivalent to ours.

For rational conductances, put $R=2Q-I$ and apply their proof to the
half-lazification $Q=(I+R)/2$ of the corresponding looped multigraph walk.
For arbitrary conductances, use the rational approximations $Q_m$ constructed
in the proof of Lemma~\ref{lem:kms-reversible-kernel}. Put
\[
        \alpha_m:=\frac{t_{\rm meet}^{Q_m}}{t_{\rm mix}^{Q_m}}.
\]
The meeting-time convergence and the bound
$t_{\rm mix}^{Q_m}\le C t_{\rm mix}^Q$ imply, for all sufficiently large $m$,
that $\alpha_m\ge c\alpha_Q$. Hence, after increasing $a_0$ and decreasing
$c_0$ by universal factors, $\alpha_Q\ge a_0$ implies both that the KMS
hypothesis holds for $Q_m$ and that
\[
        c_0\alpha_Q\log\alpha_Q
        \le c_*\,\alpha_m\log\alpha_m,
\]
where $c_*$ is the particle-count constant in \cite[Lemma~3.7]{kms}.
That lemma therefore applies to $Q_m$ for every initial configuration allowed
in the statement. Convergence of the absorption-time vectors for the finite
coalescing chains, proved in Lemma~\ref{lem:kms-reversible-kernel}, lets
$m\to\infty$ and gives the asserted bound for $Q$. The adjusted constants
$a_0,c_0,C_0$ remain universal.
\end{proof}

\begin{lemma}[Small-particle KMS comparison]\label{lem:ct-kms-small}
Let $(X_t)_{t\ge0}$ be a finite irreducible reversible continuous-time chain
with at least two states and generator $\mathcal L=P-I$, where $P$ is an
irreducible reversible transition kernel. Let
$t_{\rm meet}^{\rm ct}$ be the worst-case expected meeting time of two
independent copies and let $t_{\rm mix}^{\rm ct}$ be its total-variation mixing
time. There are universal constants $a_0,c_0,C_0\in(0,\infty)$ such that the
following holds. Put
\[
        \alpha:=\frac{t_{\rm meet}^{\rm ct}}{t_{\rm mix}^{\rm ct}}.
\]
If $\alpha\ge a_0$, then the expected coalescence time of the process started
from any configuration with at most $c_0\alpha\log\alpha$ particles is at most
$C_0t_{\rm meet}^{\rm ct}$.
\end{lemma}
\begin{proof}
Fix $\delta\in(0,\log2]$, for instance $\delta=1/2$, and put
$Q=e^{\delta\mathcal L}$. Then $Q$ is reversible and
$Q(x,x)\ge e^{-\delta}\ge1/2$. As in
\eqref{eq:mesh-mixing-full},
\begin{equation}\label{eq:mesh-mixing-comparison-small}
        t_{\rm mix}^{\rm ct}
        \le\delta t_{\rm mix}^Q
        \le t_{\rm mix}^{\rm ct}+\delta .
\end{equation}
Mesh meetings are harder than continuous-time meetings, while the restart
argument from the proof of Lemma~\ref{lem:ct-kms} gives the reverse comparison:
\begin{equation}\label{eq:mesh-meeting-comparison-small}
        t_{\rm meet}^{\rm ct}
        \le\delta t_{\rm meet}^Q
        \le C(t_{\rm meet}^{\rm ct}+\delta).
\end{equation}
Since the chain has at least two states, its worst-case meeting and mixing
times have positive universal lower bounds under the rate-one normalization.
Thus the ratios
$t_{\rm meet}^Q/t_{\rm mix}^Q$ and
$t_{\rm meet}^{\rm ct}/t_{\rm mix}^{\rm ct}$ are comparable up to universal
constants.

Let $a_0^{\rm KMS},c_0^{\rm KMS},C_0^{\rm KMS}$ be the constants in
Lemma~\ref{lem:kms-small-discrete}. After increasing $a_0$ and decreasing
$c_0$ if necessary, $\alpha\ge a_0$ implies
\[
        \frac{t_{\rm meet}^Q}{t_{\rm mix}^Q}\ge a_0^{\rm KMS},
        \qquad
        c_0\alpha\log\alpha
        \le c_0^{\rm KMS}
        \frac{t_{\rm meet}^Q}{t_{\rm mix}^Q}
        \log\!\left(\frac{t_{\rm meet}^Q}{t_{\rm mix}^Q}\right).
\]
Apply Lemma~\ref{lem:kms-small-discrete} to the mesh chain $Q$ from any
initial configuration with at most $c_0\alpha\log\alpha$ particles. Apply also
the priority coupling constructed in the proof of Lemma~\ref{lem:ct-kms},
started from this same configuration. It gives
\[
        T_{\rm coal}^{\rm ct}\le\delta T_{\rm coal}^Q
        \quad\text{pathwise}.
\]
Consequently, by \eqref{eq:mesh-meeting-comparison-small},
\[
        \E T_{\rm coal}^{\rm ct}
        \le\delta C_0^{\rm KMS}t_{\rm meet}^Q
        \le C_0t_{\rm meet}^{\rm ct}.
\]
This proves the continuous-time small-particle comparison.
\end{proof}

\begin{proof}[Proof of Corollary~\ref{cor:voter}]
By the classical voter/coalescing-random-walk duality, the consensus time with
all initial opinions distinct has the same law as the full coalescence time of
coalescing random walks started with one particle at every vertex \cite{liggett}.
We first prove a lower bound that is valid for every fixed graph in the three
regimes. Use the Harris graphical construction of the voter model \cite{harris}
and sample $U,V$ independently from the stationary measure $\pi$, independently
of the construction. The opinions at $U$ and $V$ at time $t$ are obtained by
following the two dual coalescing random walks backwards from $(U,t)$ and
$(V,t)$ to time $0$. These two backward lineages evolve as independent copies
until their first meeting. If they have not met by time $t$, then they end at
two distinct initial vertices. Since all initial opinions are distinct, the two
opinions at $U$ and $V$ are different. Therefore
\[
        \Pbb^H(\cons>t)
        \ge
        \Pbb^H_{\pi\otimes\pi}(\meet>t).
\]
Integrating in $t$ gives
\begin{equation}\label{eq:voter-lower-duality}
        \E^H\cons\ge \E^H_{\pi\otimes\pi}\meet .
\end{equation}
The lower bounds in the two supercritical parts of the corollary follow from
Theorem~\ref{thm:main-results}(i)--(ii), and the critical lower bound follows
from Theorem~\ref{thm:main-results}(iii), after decreasing the displayed
constants if necessary.

Next prove the fixed-supercritical upper bound. Let $H=H_\lambda$ and put
\[
        t_{\rm meet}(H)
        =\max_{x,y}\E^H_{x,y}\meet .
\]
Theorem~\ref{thm:main-results}(ii) gives absolute constants $c_*,C_*>0$ such
that $c_*n\le t_{\rm meet}\le C_*n$ with high probability. In the
fixed-supercritical case, $\tmix(H_\lambda)=O_\lambda(\log^2 n)$ by
Benjamini--Kozma--Wormald \cite[Theorem~1.1]{bkw}, and $|V(H_\lambda)|\le n$.
Hence, for every fixed $\lambda>1$ and all sufficiently large $n$,
\[
        \tmix(H_\lambda)\log^2 |V(H_\lambda)|\le n.
\]
Lemma~\ref{lem:ct-kms} gives
\[
        \E^{H_\lambda}\coal\le C_{\rm sup}n
\]
with an absolute $C_{\rm sup}$, after increasing the threshold in $n$ in a way
that may depend on the fixed $\lambda$.

We turn to the slightly supercritical upper bound, where no extra mixing-size
condition is imposed. Let $H=H_{\rm ER}$ and set
\[
        a:=\eps^3n.
\]
Work on the intersection of the high-probability events on which
Theorem~\ref{thm:main-results}(i), Lemma~\ref{lem:standard}, and
Lemma~\ref{lem:slightly-supercritical-hit} hold. On this event,
\begin{equation}\label{eq:slightly-supercritical-voter-inputs}
        c_*n\le t_{\rm meet}(H)\le C_*n,
        \qquad
        \tmix(H)\le C_{\rm mix}n\frac{\log^2 a}{a},
        \qquad
        t_{\rm hit}^{\rm ct}(H)\le C_{\rm hit}n\log a,
\end{equation}
with absolute constants, and $a\to\infty$. Put
\[
        \alpha:=\frac{t_{\rm meet}(H)}{\tmix(H)} .
\]
By \eqref{eq:slightly-supercritical-voter-inputs},
\[
        \alpha\ge c_\alpha\frac{a}{\log^2 a}\to\infty .
\]
Let $c_0,a_0$ be the constants from Lemma~\ref{lem:ct-kms-small}. For all large
$n$, $\alpha\ge a_0$. Define
\[
        k:=\max\left\{1,\left\lfloor\frac{c_0}{2}\alpha\log\alpha\right\rfloor\right\}.
\]
Then $k\le c_0\alpha\log\alpha$, and the preceding lower bound on $\alpha$
implies
\[
        k\ge c_k\frac{a}{\log a}
\]
for all large $n$, with an absolute $c_k>0$. Let $C_k^{\rm coal}$ be the first
time at which the coalescing process on $H$ has at most $k$ particles.
Oliveira's partial-coalescence estimate, Lemma~\ref{lem:oliveira-hit}, gives
\[
        \E^H C_k^{\rm coal}
        \le C_{\rm Oli}\left(\frac{t_{\rm hit}^{\rm ct}(H)}{k}+\tmix(H)\right)
        \le C n\frac{\log^2 a}{a}
        =o(n).
\]
At time $C_k^{\rm coal}$ at most $k\le c_0\alpha\log\alpha$ particles remain.
By the strong Markov property and Lemma~\ref{lem:ct-kms-small}, the expected
additional time to full coalescence is at most $C_0t_{\rm meet}\le C_0C_*n$.
Thus, after enlarging $C_{\rm sup}$ if necessary,
\[
        \E^{H_{\rm ER}}\coal\le C_{\rm sup}n .
\]

It remains to prove the critical-window upper bound. Fix $A_0\ge0$ and
$\eta>0$, and let $H=\mathcal C_1$ be the largest component of $G(n,p_n)$ with
$|n^{1/3}(np_n-1)|\le A_0$. Apply Lemma~\ref{lem:critical-geometry} with
parameter $\eta/2$, and write $B:=B_{A_0,\eta/2}$. On the resulting event,
\[
        |E(H)|\le B n^{2/3},
        \qquad
        \max_{x,y\in V(H)}\Reff^H(x,y)\le Bn^{1/3}.
\]
The continuous-time commute identity gives, for every $x,y\in V(H)$,
\[
        \E_x^H\tau_y+
        \E_y^H\tau_x
        =2|E(H)|\Reff^H(x,y)
        \le 2B^2 n .
\]
Since both terms on the left are nonnegative,
\[
        t_{\rm hit}^{\rm ct}(H):=\max_{x,y}\E_x^H\tau_y
        \le 2B^2 n .
\]
Lemma~\ref{lem:oliveira-hit} yields
\[
        \E^H\coal
        \le 2C_{\rm Oli}B^2 n .
\]
Apply Theorem~\ref{thm:main-results}(iii) separately with error parameter
$\eta/2$ and use \eqref{eq:voter-lower-duality} on that event. Intersecting the
two events, each of failure probability at most $\eta/2$, gives the displayed
critical-window estimate after changing constants. Duality transfers all upper
bounds from full coalescence to voter consensus.
\end{proof}

\section*{Tool and computational resource disclosure}
The authors used AI-assisted tools during the preparation of this manuscript for editorial and expository support, including improving formulations,
organization, and presentation. These tools were not used as authors and did not replace the authors' mathematical judgment. All mathematical statements, proofs, references, and final wording were checked and approved by the authors,
who take full responsibility for the content of the manuscript.

\section*{Acknowledgments}
The research of V.~Koval and P.~Trapman was supported by The Netherlands Organisation for Health Research and Development (ZonMw), grant number: 10710062310027. The research of Y. Peres was supported by National Natural Science Foundation of China grant RFIS-W2531011.


\begin{thebibliography}{99}

\bibitem{aldous-meeting}
D.~Aldous,
\newblock \textit{Meeting times for independent Markov chains},
\newblock Stochastic Processes and their Applications \textbf{38} (1991), no.~2,
185--193.
\newblock \texttt{doi:10.1016/0304-4149(91)90090-Y}.

\bibitem{af}
D.~Aldous and J.~A.~Fill,
\newblock \textit{Reversible Markov Chains and Random Walks on Graphs},
\newblock unfinished monograph, 2002; recompiled version, 2014.
\newblock Available at \texttt{https://www.stat.berkeley.edu/\string~aldous/RWG/book.html}.

\bibitem{avena-capannoli-hazra-quattropani}
L.~Avena, F.~Capannoli, R.~S.~Hazra and M.~Quattropani,
\newblock \textit{Meeting, coalescence and consensus time on random directed graphs},
\newblock Annals of Applied Probability \textbf{34} (2024), no.~5, 4940--4997.
\newblock \texttt{doi:10.1214/24-AAP2087}. arXiv:2308.01832.

\bibitem{bdnp-cover}
M.~T.~Barlow, J.~Ding, A.~Nachmias and Y.~Peres,
\newblock \textit{The evolution of the cover time},
\newblock Combinatorics, Probability and Computing \textbf{20} (2011), no.~3,
331--345.
\newblock \texttt{doi:10.1017/S0963548310000489}. arXiv:1001.0609.

\bibitem{bkw}
I.~Benjamini, G.~Kozma and N.~C.~Wormald,
\newblock \textit{The mixing time of the giant component of a random graph},
\newblock Random Structures \& Algorithms \textbf{45} (2014), no.~3, 383--407.
\newblock \texttt{doi:10.1002/rsa.20539}. arXiv:math/0610459.

\bibitem{cooper-frieze-cover}
C.~Cooper and A.~M.~Frieze,
\newblock \textit{The cover time of the giant component of a random graph},
\newblock Random Structures \& Algorithms \textbf{32} (2008), no.~4, 401--439.
\newblock \texttt{doi:10.1002/rsa.20201}. arXiv:0803.0929.

\bibitem{cooper-frieze-radzik}
C.~Cooper, A.~M.~Frieze and T.~Radzik,
\newblock \textit{Multiple random walks in random regular graphs},
\newblock SIAM Journal on Discrete Mathematics \textbf{23} (2009), no.~4,
1738--1761.
\newblock \texttt{doi:10.1137/080729542}.

\bibitem{cox}
J.~T.~Cox,
\newblock \textit{Coalescing random walks and voter model consensus times on the torus in
$\mathbb Z^d$},
\newblock Annals of Probability \textbf{17} (1989), no.~4, 1333--1366.
\newblock \texttt{doi:10.1214/aop/1176991158}.

\bibitem{dklp}
J.~Ding, J.~H.~Kim, E.~Lubetzky and Y.~Peres,
\newblock \textit{Anatomy of a young giant component in the random graph},
\newblock Random Structures \& Algorithms \textbf{39} (2011), no.~2, 139--178.
\newblock \texttt{doi:10.1002/rsa.20342}. arXiv:0906.1839.

\bibitem{dklp-diameter}
J.~Ding, J.~H.~Kim, E.~Lubetzky and Y.~Peres,
\newblock \textit{Diameters in supercritical random graphs via first passage percolation},
\newblock Combinatorics, Probability and Computing \textbf{19} (2010), no.~5--6,
729--751.
\newblock \texttt{doi:10.1017/S0963548310000301}. arXiv:0906.1840.

\bibitem{ding-lee-peres-cover}
J.~Ding, J.~R.~Lee and Y.~Peres,
\newblock \textit{Cover times, blanket times, and majorizing measures},
\newblock Annals of Mathematics (2) \textbf{175} (2012), no.~3, 1409--1471.
\newblock \texttt{doi:10.4007/annals.2012.175.3.8}. arXiv:1004.4371.

\bibitem{dlp}
J.~Ding, E.~Lubetzky and Y.~Peres,
\newblock \textit{Mixing time of near-critical random graphs},
\newblock Annals of Probability \textbf{40} (2012), no.~3, 979--1008.
\newblock \texttt{doi:10.1214/11-AOP647}. arXiv:0908.3870.

\bibitem{dlp-strict}
J.~Ding, E.~Lubetzky and Y.~Peres,
\newblock \textit{Anatomy of the giant component: The strictly supercritical regime},
\newblock European Journal of Combinatorics \textbf{35} (2014), 155--168.
\newblock \texttt{doi:10.1016/j.ejc.2013.06.004}. arXiv:1202.6112.

\bibitem{vanderhofstad}
R.~van der Hofstad,
\newblock \textit{Random Graphs and Complex Networks. Volume 1},
\newblock Cambridge Series in Statistical and Probabilistic Mathematics, vol.~43,
Cambridge University Press, Cambridge, 2017.
\newblock \texttt{doi:10.1017/9781316779422}.

\bibitem{harris}
T.~E.~Harris,
\newblock \textit{Additive set-valued Markov processes and graphical methods},
\newblock Annals of Probability \textbf{6} (1978), no.~3, 355--378.
\newblock \texttt{doi:10.1214/aop/1176995523}.

\bibitem{kms}
V.~Kanade, F.~Mallmann-Trenn and T.~Sauerwald,
\newblock \textit{On coalescence time in graphs: When is coalescing as fast as meeting?},
\newblock ACM Transactions on Algorithms \textbf{19} (2023), no.~2,
Article~18, 46~pp.
\newblock \texttt{doi:10.1145/3576900}. arXiv:1611.02460.

\bibitem{lpw}
D.~A.~Levin and Y.~Peres,
\newblock \textit{Markov Chains and Mixing Times},
\newblock 2nd ed., with contributions by E.~L.~Wilmer,
American Mathematical Society, Providence, RI, 2017.
\newblock \texttt{doi:10.1090/mbk/107}.

\bibitem{liggett}
T.~M.~Liggett,
\newblock \textit{Stochastic Interacting Systems: Contact, Voter and Exclusion Processes},
\newblock Grundlehren der Mathematischen Wissenschaften, vol.~324,
Springer, Berlin, 1999.
\newblock \texttt{doi:10.1007/978-3-662-03990-8}.

\bibitem{lyons-peres}
R.~Lyons and Y.~Peres,
\newblock \textit{Probability on Trees and Networks},
\newblock Cambridge Series in Statistical and Probabilistic Mathematics, vol.~42,
Cambridge University Press, Cambridge, 2016.
\newblock \texttt{doi:10.1017/9781316672815}.

\bibitem{nachmias-peres-critical}
A.~Nachmias and Y.~Peres,
\newblock \textit{Critical random graphs: Diameter and mixing time},
\newblock Annals of Probability \textbf{36} (2008), no.~4, 1267--1286.
\newblock \texttt{doi:10.1214/07-AOP358}. arXiv:math/0701316.


\bibitem{nash-williams}
C.~St.~J.~A.~Nash-Williams,
\newblock \textit{Random walk and electric currents in networks},
\newblock Proceedings of the Cambridge Philosophical Society \textbf{55} (1959),
181--194.
\newblock \texttt{doi:10.1017/S0305004100033879}.

\bibitem{oliveira-coalescence}
R.~I.~Oliveira,
\newblock \textit{On the coalescence time of reversible random walks},
\newblock Transactions of the American Mathematical Society \textbf{364} (2012),
no.~4, 2109--2128.
\newblock \texttt{doi:10.1090/S0002-9947-2011-05523-6}. arXiv:1009.0664.

\bibitem{oliveira}
R.~I.~Oliveira,
\newblock \textit{Mean field conditions for coalescing random walks},
\newblock Annals of Probability \textbf{41} (2013), no.~5, 3420--3461.
\newblock \texttt{doi:10.1214/12-AOP813}. arXiv:1109.5684.

\bibitem{penrose}
R.~Penrose,
\newblock \textit{A generalized inverse for matrices},
\newblock Proceedings of the Cambridge Philosophical Society \textbf{51} (1955),
406--413.
\newblock \texttt{doi:10.1017/S0305004100030401}.

\bibitem{wang-dubbeldam-vanmieghem}
X.~Wang, J.~L.~A.~Dubbeldam and P.~Van Mieghem,
\newblock \textit{Kemeny's constant and the effective graph resistance},
\newblock Linear Algebra and its Applications \textbf{535} (2017), 231--244.
\newblock \texttt{doi:10.1016/j.laa.2017.09.003}.

\end{thebibliography}
\end{document}